\documentclass[a4paper,10pt]{amsart}
\usepackage{lineno}
\usepackage[utf8]{inputenc}
\usepackage{babel} 
\usepackage{lmodern}
\usepackage{xcolor}
\usepackage{amsmath}
\usepackage{amssymb}  
\usepackage{amsthm}
\usepackage{stmaryrd}
\usepackage[hidelinks, colorlinks=false]{hyperref}
\usepackage[noabbrev,capitalise,nameinlink]{cleveref}
\usepackage{cancel}

\newtheorem{theorem}{Theorem}[section]
\newtheorem*{question*}{Question}
\newtheorem{proposition}[theorem]{Proposition}
\newtheorem{lemma}[theorem]{Lemma}
\newtheorem{corollary}[theorem]{Corollary}

\newtheorem{thmABC}{Theorem}

\newtheorem{proABC}[thmABC]{Proposition}
\newtheorem{corABC}[thmABC]{Corollary}

\theoremstyle{definition}

\newtheorem{remark}[theorem]{Remark}
\newtheorem*{notation}{Notation}

\newtheorem*{conjecture*}{Conjecture}

\newcommand{\N}{\mathbb{N}} 
\newcommand{\Z}{\mathbb{Z}}   
\newcommand{\x}{\bar{x}} 
\newcommand{\vphi}{\varphi}

\title[Engel probability in wreath products of $p$-groups]{Engel probability in wreath products of $p$-groups}

\newcommand{\commatteo}[1]{\footnotemark\marginpar{\color{blue} \tiny \footnotemark[\arabic{footnote}] #1}}

\newcommand{\noshow}[1]{}

\author[de las Heras]{Iker de las Heras}
\address{I.\ de las Heras: Matematika Saila, Euskal Herriko Unibertsitatea, 48080 Leioa, Spain}
\email{iker.delasheras@ehu.eus}
\author[Toti]{Tommaso Toti}
\address{T.\ Toti: Dipartimento di Matematica e Applicazioni, Università degli Studi di Milano-Bicocca, Via Roberto Cozzi, 55 - 20126, Milano, Italy \newline
	\indent Matematika Saila, Euskal Herriko Unibertsitatea, 48080 Leioa, Spain}
\email{tommaso.toti@gmail.com}
\author[Vannacci]{Matteo Vannacci}
\address{M.\ Vannacci: Dipartimento di Matematica `Ulisse Dini', Universit\`a degli Studi di Firenze, Viale Morgagni 67/A, 50134 Firenze, Italy}
\email{matteo.vannacci@unifi.it}


\begin{document}
	\begin{abstract}
        We give upper and lower bounds for the number of solutions of the equation $e_n(x,y) = g$ in the group $W_k=(C_p\wr C_{p^k})^2$, where $e_n(x,y)$ is the $n$-th Engel word and $g\in W_k$. We obtain several corollaries from this.
        First, we prove a stronger version of the Amit-Ashurst conjecture for Engel words in $W_k$. 
	    We also prove that Engel words are not probabilistic identities in profinite groups with arbitrarily large wreath product quotients $W_k$.
        To conclude, we construct closed subsets of $(C_p\wr\Z_p)^2$ with positive Haar measure, empty-interior, and which are the preimage of an Engel word map.
    \end{abstract}
    \maketitle
    
	\section{Introduction}
	
	\subsection{State of the art}
	
	%The study of identities in group theory has a long and fruitful history. In this article we study the probability that elements in certain finite $p$-groups (and pro-$p$ groups) satisfy an Engel word.
	
	Let $\Gamma$ be a group and let $w\in F_r$ be a non-trivial word in the free abstract group of rank $r$. One says that $w$ is an identity on $\Gamma$ if for all $g_1,\ldots,g_r \in \Gamma$ one has $w(g_1,\ldots,g_r)=1$.
	The study of identities in groups has a long tradition. For example, the restricted Burnside problem, which was one of the central problems in Algebra of the twentieth century, was centered around torsion identities and its solution involves Engel identities \cite{Z,WZ}. 
    %already W.~Burnside raised the question whether a finitely generated group of finite exponent $n$ is necessarily a finite group. 
    Several interesting classes of groups are defined by imposing an identity, including abelian, nilpotent, soluble groups, and groups of finite exponent. 
	
	Over the past decades, probabilistic methods in group theory have been used quite successfully (see, e.g., \cite{G73} and \cite{Ne89}). In particular, some difficult group-theoretical results --like the Magnus problem-- have been reproved with a much shorter probabilistic proof (see \cite{DPSS}). One natural way to introduce probability on a profinite group $G$ is to endow the group $G$ with the Haar probability measure $\mu$ making $(G,\mathcal{A},\mu)$ a probability space, where
	$\mathcal{A}$ denotes the Borel $\sigma$-algebra of $G$. For a fix $g\in G$ and a word $w$ in $k$ letters, we will denote by $P(G,w,g)$ the Haar measure of the fiber $w^{-1}(g)$ in $G^k$, and we will just write $P(G,w)$ for $P(G,w,1)$. One says that $w$ is a \emph{probabilistic identity} for $G$ if the probability $P(G,w)$ that the word $w$ is satisfied in $G$ is positive.
	
	The interplay between identities and pro\-ba\-bi\-li\-stic identities has been investigated quite intensively. In 2018 Larsen and Shalev proved that groups with infinitely many non-isomorphic non-abelian upper composition factors cannot satisfy a pro\-ba\-bi\-li\-stic identity \cite[Thm.~1.5]{LS18}; thus, one may restrict their attention to groups with restricted composition factors, e.g.\ pro-$p$ groups. %Their result also implies that finitely generated linear groups satisfy a \emph{probabilistic Tits alternative}: they are either virtually soluble or randomly free. 
	Other works investigating the relation between identities and probabilistic identities include \cite{WZ}, \cite{CM07}, \cite{BG} \cite{Ni} and \cite{LP00}.
	
	Already in 2000 Aner Shalev conjectured that a finitely generated pro-$p$ group $G$ satisfying some probabilistic identity $w$ must satisfy some identity $\tilde{w}$ (\cite[pg.~7]{Sh00}). Recently, the conjecture was proved for the class of $p$-adic analytic pro-$p$ groups. Indeed, in \cite{KOTVW}, the authors prove that a $p$-adic analytic pro-$p$ group is either virtually-soluble or it does not satisfy any probabilistic identity. On the other hand, the problem is wide open for pro-$p$ groups that are not $p$-adic analytic.
    
    By a famous criterion of Shalev \cite{Sha92}, a pro-$p$ group $G$ is $p$-adic analytic if and only if there is a $k_0\ge 0$ such that $G$ does not involve the wreath products $W_k:=C_p\wr C_{p^k}$ as closed section for all $k \ge k_0$. In light of the previous discussion, it becomes relevant to study the probability of satisfying a certain identity in the aforementioned finite $p$-groups $W_k$. 
	
	At the same time, the wreath products $W_k$ have played a central role in other investigations in recent years. For instance, regarding the theory of Hausdorff dimension, the wreath products $W_k$ are used in \cite{KT19}, \cite{dlHK22}, \cite{dlHT22-1} and \cite{dlHT22-2} to construct finitely generated pro-$p$ groups with pathological Hausdorff spectra.
    
    In the present article we study the asymptotic behaviour of the probability $P(W_k,e_n)$ for $k\to\infty$, where $e_n(x,y)$ is the $n$-th Engel word (see Section \ref{sec:prelim} for the definition of Engel word).
	
	\subsection{Results}
	
	Throughout, let $p$ be a prime. 
    For an element $h \in W_k$, $h\neq 1$, we denote by $\mathrm{deg}(h)$ the degree of $h$ with respect to the lower central series of $W_k$, that is, $\mathrm{deg}(h)=i$ if $h\in\gamma_i(W_k)\smallsetminus\gamma_{i+1}(W_k)$. Moreover, let $\gamma_\infty(W_k):=\cap_{i\geq 1}\gamma_i(W_k)$, and note that $\gamma_\infty(W_k)=1$ since $W_k$ is a nilpotent group. We use the convention $\deg(1)=\infty$. For $h\neq 1$, define $l(h)=\lceil\mathrm{log}_p(\mathrm{deg}(h))\rceil$, and note that $l(h)=l$ if and only if $h\in\gamma_{p^{l-1}+1}(W_k)\smallsetminus\gamma_{p^l+1}(W_k)$. In particular, $h\notin\gamma_{p^{l(h)}+1}(W_k)$. We also use the convention $l(1)=k$. Moreover, let $\delta_0(m)$ be the characteristic function of $0$: $\delta(0)=1$ and $\delta_0(m)=0$ for every $m\in\N$. Finally, we denote by $\delta_1(g)$ the characteristic function of $1$: $\delta_1(1)=1$ and $\delta_1(g)=0$ for $1\neq g\in W_k$.

    As will become clear later in this article, $W_k$ is nilpotent of class $p^k$ and $\mathrm{Im}(e_n)=\gamma_{n+1}(W_k)$ for all $n\in\N$. In particular, $W_k$ satisfies the Engel word $e_{p^m}$ for every $m\geq k$, that is, 
	\[ m\geq k \Rightarrow P(W_k,e_{p^m},g)=\begin{cases}
	   	1 & g=1,\\
	   	0 & g\neq 1.
	\end{cases} \] 
	Moreover, if $l(g)\leq m$, then $g\notin\gamma_{p^m+1}(W_k)$, and hence the fiber $e_{p^m}^{-1}(g)$ is empty, that is, 
    \[ l(g)\leq m \Rightarrow P(W_k,e_{p^m},g)=0. \]
    The main result of the article is to estimate the probability $P(W_k,e_{p^m},g)$ for the remaining cases.
    
	\begin{thmABC}\label{thm:engelbound_gen}
		Let $n=p^m$ with $0\le m\leq k-1$. Then 
      \begin{multline*}
     \frac{\delta_1(g)}{p^{(1+\delta_0(m))k}}+\left(1+\delta_0(m)\frac{2}{p}\right)n^2 (p-1)^2 \frac{p^{p^{l-1}-2l}}{p^{p^k}}\leq
 P(W_k,e_n,g)\\
 \leq  4\delta_1(g)\frac{n}{p^{(1+\delta_0(m))k}} + 4 n^2 (p-1)^2 \frac{(8k+4) p^{p^{l-\delta_0(m)}-2l}}{p^{p^k}}.
 \end{multline*} 
%        \[ \frac{\bar{c}(n,p)}{p^{(1+\delta_0(m))k}}\leq P(W_k,e_n)\leq \frac{\bar{C}(n,p)k}{p^{(1+\delta_0(m))k}} \]
%		for some constants $\bar{C}(n,p)\geq\bar{c}(n,p)\geq 1$\comment{Is the bar necessary?}.
%        Moreover, let $g\in W_k$, $g\neq 1$, be such that $l=l(g)\geq m+1$. Then
%		\[ \frac{c(n,p)}{p^{p^k}}\biggl(\frac{p^{p^{l-1+\delta_0(m)}}}{p^{3l}}\biggr)\leq P(W_k,e_n,g)\leq  \frac{C(n,p)k}{p^{p^k}}\biggl(\frac{p^{p^{l-\delta_0(m)}}+p^{p^{l-1}}}{p^l}\biggr),\] 
%        where $c(n,p)=n^3(p-1)^{(3-\delta_0(m))}$ and $C(n,p)=n(p-1)$.
	\end{thmABC}

	When $g=1$ it is not difficult to extend \cref{thm:engelbound_gen} to every $n\geq 1$. Indeed, let $m=\lfloor\mathrm{log}_p(n)\rfloor$ and $M=\lceil\mathrm{log}_p(n)\rceil$. Since $p^m\leq n\leq p^M$, we have $e_{p^m}^{-1}(1)\subseteq e_n^{-1}(1)\subseteq e_{p^M}^{-1}(1)$, and hence \[ P(W_k,e_{p^m})\leq P(W_k,e_n)\leq P(W_k,e_{p^M}). \] 
    
    %On the other hand, it is not clear if the previous result can be extended to every $n\geq 1$ for $g\neq 1$.
    %\comment{Really?}\comtommaso{If $g\neq 1$, $e_{p^m}^{-1}(g)\subseteq e_n^{-1}(g)\subseteq e_{p^M}^{-1}(g)$ with $p^m\le n\leq p^M$ does not hold}

    Given a word $w$ and a finite group $G$, one might expect that the fiber $w^{-1}(1_G)$ of the identity is always the largest among all fibers. This is not true in general; for instance, for the word $w(x)=x^2$ on the quaternion group $Q_8$ of order 8 and the  non-trivial central element $z$ of $Q_8$, we have that $P(Q_8,x^2,z) = \frac{3}{4}$ and $P(Q_8,x^2) = \frac{1}{4}$. %On the other hand, when this happens it is useful because we automatically can obtain an analogous result to Theorem~\ref{thm:engelbound} for any fiber. 
    Along the way to proving \cref{thm:engelbound_gen}, it can also be deduced that the fiber $e_{n}^{-1}(1)$ is the largest among all fibers. More precisely, we prove the following.

	\begin{proABC}\label{prop:biggest_fiber}
		Let $n\in\N$ and let $g_1,g_2\in W_k$ be such that $\mathrm{deg}(g_1)\leq\mathrm{deg}(g_2)$. Then \[ P(W_k,e_n,g_1)\leq P(W_k,e_n,g_2). \] In particular, $P(W_k,e_n,g)\leq P(W_k,e_n)$ for every $g\in W_k$.
	\end{proABC}

	%The groups $W_k$ fit into an inverse system and their inverse limit is the group $W=C_p \wr \mathbb{Z}_p$. Our main application of \cref{thm:engelbound} is that Engel words cannot be probabilistic identities in pro-$p$ group $W$.

        To boot, the information obtained during the proof of Theorem~\ref{thm:engelbound_gen} is so precise that we can confirm the Amit-Ashurst conjecture for certain Engel words. In 2012 Ashurst \cite{Ash} conjectured that the size of a generic fiber for any word map in a finite $p$-group $G$ is greater than $1/\lvert G \rvert$; that is, for a word in $k$ letters, $\lvert w^{-1}(g) \rvert \ge \lvert G \rvert^{k-1}$. In terms of probability, this conjecture can be restated as $P(G,w,g)\geq 1/|G|$, for any $g$ in the image of $w$. Given $H\leq G$, it is clear that, if $P(G,w,g)\geq 1/|H|$ for any $g$ in the image of $w$, then the Amit-Ashurst conjecture holds for the group $G$ and the word $w$.

        %In the following corollary, we prove a stronger version of the Amit-Ashurst conjecture for the Engel word $e_{p^m}$ in the groups $W_k$.

		%\begin{corABC}\label{thm:amit_ashurts}
		%Let $n=p^m$, and let $g\in\mathrm{Im}(e_n)$.
        %Then 
	   % \[ P(W_k,e_n,g)\geq \frac{1}{|B_k|}=\frac{p^k}{|W_k|},\]
		%where $B_k$ is the base group of $W_k$. In particular, $W_k$ satisfies the Amit-Ashurst conjecture for the Engel word $e_{p^m}$. 
		%\end{corABC}
		%This easily follows from the proof of Theorem~\ref{thm:engelbound_gen} (see in particular Lemma~\ref{lemma:upperboundsS_ig}). 

         We prove the following result (we use the convention $\mathbb{N}_0 = \mathbb{N} \cup \{0\}$).
        
        \begin{corABC}\label{thm:amit_ashurts_gen}
        Let $n=p^m$, with $m\in\N_0$, and let $g\in\mathrm{Im}(e_n)$.
        Then 
        \[ P(W_k,e_n,g)\geq \left(1+\delta_0(m)\frac{2}{p}\right)\biggl(1-\frac{1}{p}\biggr)^2\frac{1}{|\mathrm{Im}(e_n)|}\geq\frac{1}{|B_k|}, \]
        where $B_k$ is the base group of $W_k$. In particular, the Amit-Ashurst conjecture is satisfied for the group $W_k$ and the Engel word $e_{p^m}$.
        
        \end{corABC}

   To conclude, we present some applications to the theory of pro-$p$ groups. Let $W$ be the inverse limit of the groups $W_k$. Then $W\simeq C_p\wr\Z_p$ is a $2$-generator pro-$p$ group of infinite rank. In particular, $W$ is a finitely generated pro-$p$ group that is not $p$-adic analytic.

   As a first application, since the upper bound in \cref{thm:engelbound_gen} for $P(W_k,e_n,g)$ tends to 0 as $k\to\infty$, we immediately have that $W$ does not satisfy Engel probabilistic identities. More precisely, we have the following.
   \begin{corABC}\label{cor:cpZp}
			Let $G$ be a profinite group which projects onto $W_k$ for every $k$. Then 
            \[ P(G,e_n,g)=0 \]
            for every $g\in G$ and every $n\in\N$.
	\end{corABC}
   In particular, choosing $G=W$ and $g=1$, we obtain $P(W,e_n)=0$ for every $n\in\N$, that is, Engel words are not probabilistic identities on $W$.
   In view of \cite[Problem 3.1]{LS16} and the aforementioned results in \cite{Sha92} and \cite{KOTVW}, we rise the following question.
    \begin{question*}
    Let $G$ be a profinite group, and suppose that, for every $k\in\N$, there exist $K,H\le G$ with $K\trianglelefteq H$ such that $H/K\cong W_k$. Then $P(G,e_n,g)=0$ for every $g\in G$ and every $n\in\N$.
   \end{question*}
   A positive answer to the previous question would provide a positive solution to \cite[Problem 3.1]{LS16} for Engel words in the class of pro-$p$ groups.  
   %The previous result relies on the upper bound obtained in \cref{thm:engelbound_gen}, the following application of the lower bound estimated in the aforementioned theorem. 
   %together with some probabilistic (\cref{prop:engelmeasure}) and topological (\cref{prop:criteriumemptyinterior2}) properties of the word map $e_n\colon W^2\to W$, we prove the following.  

   Finally, through the study of certain probabilistic (\cref{prop:engelmeasure}) and top\-o\-log\-i\-cal (\cref{prop:criteriumemptyinterior2}) properties of the word map $e_n\colon W^2\to W$, we obtain the following result.
   \begin{corABC}\label{cor:construction}
		Let $n=p^m$ with $m\in\N_0$. Let $S$ be a closed subset of $I_n=\mathrm{Im}(e_n)$ such that 
		\begin{enumerate}
			\item[(i)] $\mu_{I_n}(S)>0$, where $\mu_{I_n}$ is the normalized Haar measure on $I_n$;
			\item[(ii)] $S$ has empty-interior as a subset of $I_n$.
		\end{enumerate}
	    Then $e_n^{-1}(S)$ is a closed subset of $W^2$ of positive Haar measure and empty-interior.
	\end{corABC}
    The existence of closed subsets of profinite groups with positive Haar measure and empty-interior is known, for example the subset constructed in \cite{LP00}. How\-ev\-er, to the best of our knowledge, \cref{cor:construction} provides the first examples of such closed subsets which are the preimage of a word map.
   
   %we prove that certain closed subsets given a closed subset $X$ of $W$, we prove that $e_n^{-1}(X)$ has positive Haar measure and empty topological interior (\cref{cor:construction}).  }
    
	%{\color{red} We need to decide if to include (if anything) from the automorphism group... If exclude, remmeber to remove the last corollary}
	
	\smallskip
	
	\noindent \textbf{Organisation of the article.} Section~\ref{sec:prelim} is devoted to some preliminaries. In Section~\ref{commsection}, we study some aspects of commutators and of the lower central series of the group $W_k$. In Section~\ref{engelsection}, by means of various counting arguments, we obtain upper and lower bounds for $P(W_k,e_n, g)$ (\cref{thm:engelbound_gen}) and deduce \cref{prop:biggest_fiber} and \cref{thm:amit_ashurts_gen}. Finally, in Section~\ref{pro-psection}, we give some applications to the theory of profinite and pro-$p$ groups (\cref{cor:construction}).
	
	\section*{Acknowledgements}
	
	The first author acknowledges the support of the Basque Government grant IT1913-26. The second author acknowledges the Joint PhD Program in Math\-e\-mat\-ics Milano Bicocca - Pavia - INdAM. The third author is funded by the Italian program Rita Levi Montalcini for young researchers, Edition 2021. Some of this work was part of the second author's PhD thesis at the University of Milano-Bicocca and UPV/EHU.

	\section{Commutator calculus}\label{sec:prelim}
	
	In this short section, we collect some tools from commutator calculus that will be used frequently in this article. 
	Given a group $G$ and elements $x,y\in G$, let \[[x,y]:=x^{-1}y^{-1}xy\]
    be their commutator. If $z\in G$, we use the left-normed convention for the com\-mu\-ta\-tor $[x,y,z]$, that is \[ [x,y,z]:=[[x,y],z]. \] Moreover, we use the notation
	\[ [x,_1 y] = [x,y], \text{\quad  and \quad} [x,_ny]:=[[x,_{n-1} y],y] \text{\quad for $n\ge 2$}. \]
	We will repeatedly use the following well-known commutator formulas:
	\begin{align}\label{eq:commcalc}
		&[x,y]=[y,x]^{-1}, \nonumber\\
		&[x,yz]=[x,z][x,y]^z,\quad [xy,z]=[x,z]^y[y,z],\\
		&[x,y^{-1}]=[y,x]^{y^{-1}},\quad [x^{-1},y]=[y,x]^{x^{-1}}.\nonumber
	\end{align}

	  Next we recall a standard commutator identity; compare Lemma 2.1 of \cite{dlHK22} and Proposition 1.1.32 of \cite{LGM02}.

    \begin{lemma}\label{lemma:strongHP}
		Let $G=\langle x,y\rangle$ be a finite $p$-group such that $\gamma_2(G)$ has exponent $p$, and let $n\in\N$. For $u,v\in G$, let $K(u,v)$ denote the normal closure in $G$ of all commutators in $\{u,v\}$ of weight at least $p^n$ that have weight at least $2$ in $v$. Then: \[ [y,x^{p^n}]\equiv [y,_{p^n}x]\pmod{ K(x,[x,y])}.    \]       
	\end{lemma}

    We conclude with an elementary group-theoretical lemma.
	
	\begin{lemma}\label{lemma:inclusions}
		Let $G$ be a finite group, and let $A,B\leq G$ be such that $A\leq B$ and $x\in G$. Then \[xA\cap B=\begin{cases}
			xA & x\in B,\\
			\emptyset & x\notin B,
		\end{cases} \quad\quad\quad\quad A\cap xB=\begin{cases}
			A & x\in B,\\
			\emptyset & x\notin B.
		\end{cases}
		\]
		In particular, \[ |xA\cap B|=|A\cap xB|=
		\begin{cases}
			|A| & x\in B, \\
			0 & \text{otherwise}.
		\end{cases} \] 
	\end{lemma}
	% \begin{proof}
	% 	If $x\in B$, then $xA\subseteq xB=B$ and hence $xA\cap B=xA$. On the other hand, if $x\notin B$, then $xA\subseteq xB$. Since $xB\cap B=\emptyset$, in particular we have $xA\cap B=\emptyset$. The statement regarding $A\cap xB$ follows similarly. 
	% \end{proof}

	\section{Commutators in the finite \texorpdfstring{$p$}{p}-group \texorpdfstring{$W_k$}{Wk}}\label{commsection}

	For $k\in\N$, %We refer the reader to \cite[P. 47]{DM96} for the definition of the standard wreath product. 
	recall that the standard wreath product $W_k=C_p\wr C_{p^k}$ can be seen as the semidirect product $B\rtimes Q$, where the \emph{top group} $Q=\langle w\rangle$ is a cyclic group of order $p^k$, the \emph{base group} $B$ is the set $\mathrm{Fun}(Q,C_p)$ of functions from $Q$ to the cyclic group of order $p$ (it can be seen as an elementary abelian $p$-group of rank $p^k$) and $Q$ acts via automorphisms on $B$ by $$ x^{w^i}(w^{j}) = x(w^{i+j}), \quad \text{for $x\in B$ and $w^i,w^j\in Q$}; $$  i.e., $Q$ acts as right-multiplication on the domain of functions in $B$.
	
	\begin{notation} Fixing a generator $\langle a \rangle$ of the cyclic group of order $p$, we define $\x \in B$ by $\x(1) =a$ and $\x(y)=1$ for $1\neq y\in Q$. Then, $W_k=\langle \x,w \rangle$. We will keep this notation for the generators of $W_k$ for the rest of the article.
	\end{notation}
	
	Regarding the lower central series of the finite $p$-group $W_k$, the following holds.
	
	\begin{lemma}[{\cite[Proposition 2.6]{KT19}}]\label{lemma:lcs}
		The wreath product $W_k$ is nilpotent of class $p^k$. Moreover, the lower central series of $W_k$ satisfies
		\[ W_k=\gamma_1(W_k)=\langle \x,w \rangle\gamma_2(W_k) \quad \text{with} \quad W_k/\gamma_2(W_k)\simeq C_{p}\times C_{p^k} \]
		and 
		\[ \gamma_i(W_k)=\langle [\x,_{i-1}w]\rangle\gamma_{i+1}(W_k)\quad\text{with}\quad\gamma_i(W_k)/\gamma_{i+1}(W_k)\simeq C_p \] for $i=2,\dots,p^k$. In particular, the base group $B$ satisfies 
		\[ B=\langle\x\rangle\gamma_2(W_k)=\langle \x,[\x,w],[\x,_2 w],\dots,[\x,_{p^k-1}w] \rangle \] 
		and 
		\[ \gamma_i(W_k)=\langle [\x,_{i-1}w],[\x,_iw],\dots,[\x,_{p^k-1}w] \rangle \]
		for every $i=2,\dots, p^k$.
	\end{lemma}
	
	Now, given $t\in\N_0$, define the map 
	\begin{align*}
		\vphi_t\colon & B\to B\\
		&x\to [x,w^t].
	\end{align*}
	Note that $\mathrm{Im}(\vphi_t)\subseteq B$ since $B\unlhd G$. The following lemma collects some key properties of the map $\vphi_t$.
	
	\begin{lemma}\label{lemma:factorization}
		Let $t\in\N_0$. Then, the map $\vphi_t\colon B\to B$ is a group homomorphism. Moreover, if $t=sp^j$ with $\mathrm{gcd}(s,t)=1$, we have \[ \vphi_t=(\vphi_s)^{p^j}, \] that is $[x,w^t]=[x,_{p^j}w^s]$ for every $x\in B$. 
	\end{lemma}
	\begin{proof}
		Let $x,y\in B$. We have \[ \vphi_t(xy)=[xy,w^t]=[x,w^t]^y[y,w^t]=\vphi_t(x)\vphi_t(y), \] where the equality $[x,w^t]^y=[x,w^t]$ holds since $[x,w^t]$ and $y$ belong to the abelian group $B$.
        Therefore, $\vphi_t\colon B\to B$ is a group homomorphism. Now suppose that $t=sp^j$ with $\mathrm{gcd}(s,p)=1$. Since $\mathrm{exp}(\gamma_2(W_k))=p$, by \cref{lemma:strongHP} we have \[ [x,w^t]=[x,(w^s)^{p^j}]\equiv [x,_{p^j}w^s]\pmod{K(w^s,[w^s,x])},\] 
        %Using that $\vphi_s$ is a group homomorphism, we obtain \[ [w^s,x,_{p^j-1}w^s]=[[x,w^s]^{-1},_{p^j-1}w^s]=\vphi_s^{p^j-1}([x,w^s]^{-1})= \vphi_s^{p^j}(x)^{-1}=[x,_{p^j}w^s]^{-1}. \] Then, \[ [x,w^t]=[w^t,x]^{-1}\equiv [x,_{p^j}w^s] \pmod{K(w^s,[w^s,x])}. \] 
        
        where $K(w^s,[w^s,x])$ denotes the normal closure in $W_k$ of all commutators in $\{w^s,[w^s,x]\}$ of weight at least $p^j$ that have weight at least $2$ in $[w^s,x]$.
        Since $[w^s,x]\in \gamma_2(W_k)$ and $\gamma_2(W_k)$ is abelian, it follows that $K(w^s,[w^s,x])=1$, which yields the result. 
	\end{proof}

	Before going on with our analysis, we introduce some notation. Let $\nu$ be the $p$-valuation map, that is the map that associates to every positive integer $n=sp^j$ with $\mathrm{gcd}(s,p)=1$ the value $\nu(n)=j$. We use the convention that $\nu(0)=\infty$. Moreover, given $t_1,\dots, t_m\in\N_0$, let
    \[ v(t_1,\dots,t_m)=\sum_{i=1}^{m}p^{\nu(t_i)}, \] 
    with $v(t_1,\dots,t_m)=\infty$ whenever $t_i=0$ for some $i=1,\dots,m$. 
    Finally, define the group homomorphism $\vphi_{t_1,\dots,t_m}=\vphi_{t_m}\circ\cdots\circ\vphi_{t_1}$, that is
	\begin{align*}
		\vphi_{t_1,\dots,t_m}\colon B &\to B\\
		x&\mapsto [x,w^{t_1},\ldots,w^{t_m}].
	\end{align*}

	%\begin{proposition}\label{pr:degree}
	%	Let $t_1,\dots,t_m\in\N_0$. Then, for every $i=2,\dots,p^k$, we have:
	%	\begin{itemize}\setlength{\itemsep}{1em}
	%		\item[$(i)$] $\displaystyle \vphi_{t_1,\dots,t_m}(B)=\gamma_{v(t_1,\dots,t_{m})+1}(W_k)$\comment{If $t=0$, the notation does not work} and
	%		\item[$(ii)$] $\displaystyle \vphi_{t_1,\dots,t_m}(\gamma_i(W_k))=\gamma_{v(t_1,\dots,t_{m})+i}(W_k).$
	%	\end{itemize} 
	%\end{proposition}
    \begin{proposition}\label{pr:degree}
		Let $t_1,\dots,t_m\in\N_0$. Then
        \[ \vphi_{t_1,\dots,t_m}(B)=\gamma_{v(t_1,\dots,t_{m})+1}(W_k), \]
        and
        \[ \vphi_{t_1,\dots,t_m}(\gamma_i(W_k))=\gamma_{v(t_1,\dots,t_{m})+i}(W_k) \] 
        for every $i= 2,\dots, p^k$. Moreover, if $x\in B$, then 
        \[ \mathrm{deg}(\vphi_{t_1,\dots,t_m}(x))=\begin{cases}
            \infty & \vphi_{t_1,\dots,t_m}(x)=1, \\
            \mathrm{deg}(x)+v(t_1,\dots,t_{m}) & \text{otherwise}.
        \end{cases}\]
        
	\end{proposition}
	\begin{proof}
		We prove the result by induction on $m$. If $m=1$, let $t=t_1$. By \cref{lemma:factorization}, $\vphi_t$ is a group homomorphism. Moreover, if $t=sp^j$ with $\mathrm{gcd}(s,p)=1$, we have $[x,w^t]=[x,_{p^j}w^s]$ by \cref{lemma:strongHP}. Note that $w^s$ is a generator of the top group $Q$ as $s$ is coprime to $p$. Therefore, by \cref{lemma:lcs}, we have
        \[ \vphi_t(B)= \vphi_t(\langle \x,[\x,w^s],\dots,[\x,_{p^k-1}w^s] \rangle)=\langle [\x,_{p^j}w^s],[\x_{p^j+1}w^s],\dots,[\x,_{p^k-1}w^s] \rangle, \]
        and 
        \[ \vphi_t(\gamma_{i}(W_k))=\vphi_t(\langle [\x,_{i-1}w^s],\dots,[\x,_{p^k-1}w^s] \rangle)= \langle [\x,_{p^j+i-1}w^s],\dots,[\x,_{p^k-1}w^s] \rangle. \]
        Also, take $x\in B$ and assume that $\vphi_{t}(x)\neq 1$. By \cref{lemma:lcs}, if $d=\mathrm{deg}(x)$, we have  
        $x\gamma_{d+1}(W_k)=([\x,_{d-1}w^s]^a)\gamma_{d+1}(W_k)$ for some $a=1,\dots,p-1$. Therefore, 
        \[ \vphi_t(x\gamma_{d+1}(W_k))=([\x,_{p^j+d-1}w^s]^a)\gamma_{p^j+d+1}(W_k), \] and hence the result since $\vphi_t(x\gamma_{d+1}(W_k))=\vphi_t(x)\gamma_{p^j+d+1}(W_k)$.

        %Since 
		%\[ \vphi_t(\langle \x,[\x,w^s],\dots,[\x,_{p^k-1}w^s] \rangle)=\langle [\x,_{p^j}w^s],[\x_{p^j+1}w^s],\dots,[\x,_{p^k-1}w^s] \rangle \]
		%and 
		%\[ \vphi_t(\langle [\x,_{i-1}w^s],\dots,[\x,_{p^k-1}w^s] \rangle)=\langle [\x,_{p^j+i-1}w^s],\dots,[\x,_{p^k-1}w^s] \rangle, \] the result follows from \cref{lemma:lcs}. 
        
        Now, suppose that $m\geq2$. By the inductive hypothesis, 
		\[ \vphi_{t_1,\dots,t_{m-1}}(B)=\gamma_{v(t_1,\dots,t_{m-1})+1}(W_k), \] so that 
		\[ \vphi_{t_1,\dots,t_m}(B)=\vphi_{t_m}(\gamma_{v(t_1,\dots,t_{m-1})+1}(W_k))=\gamma_{v(t_1,\dots,t_{m})+1}(W_k) \] by the base case. The other claims can be proved analogously.
	\end{proof}
	
	In the previous proposition, we have determined the image of the group ho\-mo\-mor\-phism $\vphi_{t_1,\dots,t_m}$, that is \[ \mathrm{Im}(\vphi_{t_1,\dots,t_m})=\vphi_{t_1,\dots,t_m}(B)=\gamma_{v(t_1,\dots,t_m)+1}(W_k). \]
	For better citability, we collect in the following corollary both the description of the image and the kernel of $\vphi_{t_1,\dots,t_m}$. If $i\leq 1$, we use the convention $\gamma_i(W_k)=W_k$. 
	
	\begin{corollary}\label{corollary:kerim}
		Let $t_1,\dots,t_m\in\N_0$. Then 
		\begin{itemize}\setlength{\itemsep}{1em}
			\item[$(i)$] $\displaystyle \mathrm{Ker}(\vphi_{t_1,\dots,t_m})=\gamma_{p^k+1-v(t_1,\dots,t_m)}(W_k)\cap B$ and
			\item[$(ii)$] $\displaystyle \mathrm{Im}(\vphi_{t_1,\dots,t_m})=\gamma_{v(t_1,\dots,t_m)+1}(W_k).$
		\end{itemize} 
	\end{corollary}
	\begin{proof}
		Let $1\neq x\in B$ and suppose that $\deg(x)=i$. By \cref{pr:degree} and the fact that $W_k$ is nilpotent of class $p^k$, we have \[ [x,w^{t_1},\dots,w^{t_m}]=1\iff i+v(t_1,\dots,t_m)\geq p^k+1,  \] that is, if and only if $x\in\gamma_{p^k+1-v(t_1,\dots,t_m)}(W_k)$. 
	\end{proof}
	
	To conclude, in the following corollary we collect some counting results.
	
	\begin{corollary}\label{cor:countingkerim}
		Let $i=1,\dots,p^k+1$. Then \[ |\gamma_i(W_k)|=
		\begin{cases}
			p^{k+p^k} & i=1,\\
			p^{p^k+1-i} & i=2,\dots, p^k+1.
		\end{cases} \]
		In particular, given $t_1,\dots,t_m\in\N_0$, we have
		\[ |\mathrm{Ker}(\vphi_{t_1,\dots,t_m})|= 
		\begin{cases}
			p^{v(t_1,\dots,t_m)} & v(t_1,\dots,t_m)\leq p^k,\\
			p^{p^k} & \text{otherwise}
		\end{cases} \]
		and
		\[ |\mathrm{Im}(\vphi_{t_1,\dots,t_m})|=\frac{p^{p^k}}{|\mathrm{Ker}(\vphi_{t_1,\dots,t_m})|}. \]
	\end{corollary}
	\begin{proof}
		The results regarding $|\gamma_i(W_k)|$ and $|\mathrm{Ker}(\vphi_{t_1,\dots,t_m})|$ follow directly from Lem\-ma \ref{lemma:lcs} and \cref{corollary:kerim}. To calculate the order of $\mathrm{Im}(\vphi_{t_1,\dots,t_m})$, note that \[ \vphi_{t_1,\dots,t_m}\colon \frac{B}{\mathrm{Ker}(\vphi_{t_1,\dots,t_m})}\to \mathrm{Im}(\vphi_{t_1,\dots,t_m}) \] is a group isomorphism and $|B|=p^{p^k}$.
	\end{proof}

	\section{Proof of Theorem~\ref{thm:engelbound_gen}}\label{engelsection}
	
	Given an element $g\in W_k$, in this section we will give an upper and lower bound for the size of the fiber \[ |X(W_k, e_n, g)| := |e_n^{-1}(g)| = |\{ (g_1,g_2)\in W_k^2\mid e_n(g_1,g_2)=g \}|,\]
    that is, we will give an upper and lower bound for the number of solutions in $W_k^2$ of the equation \[ e_n(x,y)=[x,_ny]=g. \]
    Using these bounds (cfr.\ \cref{lemma:boundsS0} and \cref{lemma:boundsS1}), we will easily derive Theorem~\ref{thm:engelbound_gen} and Corollary~\ref{thm:amit_ashurts_gen}.
	
	To begin our investigation, note that every element $h\in W_k$ can be written uniquely as $h=w^tx$ with $x\in B$ and $t=0,\dots,p^k-1$. From now on, let $T=\{ 0,\dots,p^k-1 \}$. A key ingredient for counting the number of solutions in $W_k^2$ of the equation $e_n(x,y)=g$ is that, given $g_1,g_2\in W_k$, we can write $e_n(g_1,g_2)$ as a product of two commutators as shown in the following lemma. 
	
	\begin{lemma}\label{lemma:decomposition}
		Let $x_1,x_2\in B$ and $t_1,t_2\in T$. Then \[ e_n(w^{t_1}x_1,w^{t_2}x_2)=[x_1,_nw^{t_2}][x_2,w^{t_1},_{n-1}w^{t_2}]^{-1}. \]
	\end{lemma}
	\begin{proof}
		We prove the result by induction on the length $n$ of the Engel word $e_n$. 
		Since $B\unlhd W_k$ is abelian, using appropriate commutator formulas, we obtain the case $n=1$: 
		\begin{align*}
			[w^{t_1}x_1,w^{t_2}x_2] &=[w^{t_1}x_1,x_2][w^{t_1}x_1,w^{t_2}]^{x_2}\\
			&=[w^{t_1},x_2]^{x_1}[x_1,x_2][w^{t_1},w^{t_2}]^{x_1}[x_1,w^{t_2}]=[x_1,w^{t_2}][x_2,w^{t_1}]^{-1}.
		\end{align*}
		Suppose that $n\geq2$. Denoting by $c_{n-1}:=e_{n-1}(w^{t_1}x_1,w^{t_2}x_2)\in B$, we have \[ e_n(w^{t_1}x_1,w^{t_2}x_2)=[c_{n-1},w^{t_2}x_2]=[c_{n-1},x_2][c_{n-1},w^{t_2}]^{x_2}=[c_{n-1},w^{t_2}]. \] By the inductive hypothesis and \cref{lemma:factorization}, we obtain 
		\begin{align*}
			[c_{n-1},w^{t_2}]&=[[x_1,_{n-1}w^{t_2}][x_2,w^{t_1},_{n-2}w^{t_2}]^{-1},w^{t_2}]\\
			&=[x_1,_nw^{t_2}][[x_2,w^{t_1},_{n-2}w^{t_2}]^{-1},w^{t_2}]=[x_1,_nw^{t_2}][x_2,w^{t_1},_{n-1}w^{t_2}]^{-1}.\qedhere
		\end{align*} 
	\end{proof}
	
	Given $t\in\N_0$, recall that $\vphi_t^m=\vphi_{t,\overset{m}{\cdots},t}$. In view of the previous lemma, \[ e_n(w^{t_1}x_1,w^{t_2}x_2)=g \Leftrightarrow \vphi_{t_2}^n(x_1)=g(\vphi_{t_2}^{n-1}\circ\vphi_{t_1}(x_2)). \]
	Therefore, if \[ S_{t_1,t_2}(n,g):=\{ (x_1,x_2)\in B^2\mid \vphi_{t_2}^n(x_1)=g(\vphi_{t_2}^{n-1}\circ\vphi_{t_1}(x_2)) \}, \] then
	\begin{equation}\label{eq:sum_g}
		|X(W_k,e_n,g)|=|\{(g_1,g_2)\in W_k^2\mid e_n(g_1,g_2)=g\}|=\sum_{t_1,t_2\in T} s_{t_1,t_2}(n,g), 
	\end{equation} 
	where $s_{t_1,t_2}(n,g)=|S_{t_1,t_2}(n,g)|$. Since $\mathrm{Im}(e_n)=\gamma_{n+1}(W_k)$ by Lemma \ref{pr:degree}, we have $|X(W_k,e_n,g)|=0$ whenever $\mathrm{deg}(g)\leq n$, and hence $s_{t_1,t_2}(g)=0$ in this case.	
	In the following lemma, we calculate the numbers $s_{t_1,t_2}(n,g)$ in general.
	
	\begin{lemma}\label{lemma:countingSg}
		Let $t_1,t_2\in T$. Then $s_{t_1,t_2}(n,g)$ equals 
        \[ 
		\begin{cases}
			|B| \cdot |\mathrm{Ker}(\vphi_{t_2}^{n-1}\circ\vphi_{t_1})| &  \text{if } \nu(t_1)\le\nu(t_2)\text{ and } \mathrm{deg}(g)\geq (n-1)p^{\nu(t_2)}+p^{\nu(t_1)}+1,	\\
			|B| \cdot |\mathrm{Ker}(\vphi_{t_2}^n)| & \text{if } \nu(t_1)>\nu(t_2) \text{ and } %g\in\gamma_{np^{\nu(t_2)}},\\
			\mathrm{deg}(g)\geq np^{\nu(t_2)}+1, \\
			0 & \text{otherwise}.
		\end{cases} \]
        
		In particular, $s_{t_1,t_2}(n,g)$ depends only on $\nu(t_1),$ $\nu(t_2)$, $\mathrm{deg}(g)$ and $n$. 
	\end{lemma}
    %{\color{violet}
    %\[ 
	%	\begin{cases}
	%		|B| \cdot |\mathrm{Ker}(\vphi_{t_2}^{n-1}\circ\vphi_{t_1})| &  \text{if } \nu(t_1)\le\nu(t_2)\text{ and } \mathrm{deg}(g)\geq \min\{(n-1)p^{\nu(t_2)}+p^{\nu(t_1)},p^k\}+1,	\\
	%		|B| \cdot |\mathrm{Ker}(\vphi_{t_2}^n)| & \text{if } \nu(t_1)>\nu(t_2) \text{ and } %g\in\gamma_{np^{\nu(t_2)}},\\
	%		\mathrm{deg}(g)\geq \min\{np^{\nu(t_2)},p^k\}+1, \\
	%		0 & \text{otherwise}.
	%	\end{cases} \]
    %}
	\begin{proof}
        Let  $I_g(t_1,t_2)=\mathrm{Im}(\vphi_{t_2}^n)\cap g\mathrm{Im}(\vphi_{t_2}^{n-1}\circ\vphi_{t_1})$. Note that $S_{t_1,t_2}(n,g)$ is the disjoint union
        \[ S_{t_1,t_2}(n,g)=\hspace{-3pt}\bigsqcup_{x\in I_g(t_1,t_2)}\hspace{-2pt}\Bigl\{ (x_1,x_2)\in B^2\mid x_1\in (\vphi_{t_2}^n)^{-1}(x), x_2\in(\vphi_{t_2}^{n-1}\circ\vphi_{t_1})^{-1}(g^{-1}x) \Bigr\}. \]
		In general, since by Lemma \ref{lemma:factorization} $\varphi_t$ is a group homomorphism for any $t$, if $y\in\mathrm{Im}(\vphi_t)$ then $|\vphi_t^{-1}(y)|=|\mathrm{Ker}(\vphi_t)|$. Therefore, \begin{equation}\label{eq:countingSg1}
			s_{t_1,t_2}(n,g)=|\mathrm{Im}(\vphi^n_{t_2})\cap g\mathrm{Im}(\vphi_{t_2}^{n-1}\circ\vphi_{t_1})| \cdot |\mathrm{Ker}(\vphi_{t_2}^n)| \cdot |\mathrm{Ker}(\vphi_{t_2}^{n-1}\circ\vphi_{t_1})|.
		\end{equation}  
		Suppose that $\nu(t_1)\le\nu(t_2)$. By Corollary \ref{corollary:kerim}, $\mathrm{Im}(\vphi^n_{t_2})\leq\mathrm{Im}(\vphi^{n-1}_{t_2} \circ\vphi_{t_1})$, so that \begin{equation}\label{eq:countingSg2}
			|\mathrm{Im}(\vphi^n_{t_2})\cap g\mathrm{Im}(\vphi_{t_2}^{n-1}\circ\vphi_{t_1})|=\begin{cases}
				|\mathrm{Im}(\vphi^n_{t_2})| & g\in \mathrm{Im}(\vphi_{t_2}^{n-1}\circ\vphi_{t_1}) \\
				0 & \text{otherwise.}
			\end{cases}
		\end{equation}  
		Analogously, if $\nu(t_1)>\nu(t_2)$, we have 
		\begin{equation}\label{eq:countingSg3}
			|\mathrm{Im}(\vphi^n_{t_2})\cap g\mathrm{Im}(\vphi_{t_2}^{n-1}\circ\vphi_{t_1})|=\begin{cases}
				|\mathrm{Im}(\vphi^{n-1}_{t_2}\circ\vphi_{t_1})| & g\in \mathrm{Im}(\vphi_{t_2}^n) \\
				0 & \text{otherwise.}
			\end{cases}
		\end{equation}
		Combining Equations (\ref{eq:countingSg1}), (\ref{eq:countingSg2}) and (\ref{eq:countingSg3}), \cref{lemma:inclusions} and \cref{cor:countingkerim} the result follows.
	\end{proof}
	
	\begin{proof}[Proof of Proposition~\ref{prop:biggest_fiber}]
		From \cref{lemma:countingSg}, we see that $s_{t_1,t_2}(n,g_1)\leq s_{t_1,t_2}(n,g_2)$ for every $t_1,t_2\in T$. Therefore, \[ |X(W_k,e_n,g_1)|=\sum_{t_1,t_2\in T} s_{t_1,t_2}(n,g_1)\leq\sum_{t_1,t_2\in T} s_{t_1,t_2}(n,g_2)=|X(W_k,e_n,g_2)|.  \qedhere\] 
	\end{proof}

For later proofs it will be useful to have the results from Lemma~\ref{lemma:countingSg} in terms of the valuations: $ s_{t_1,t_2}(n,g)$ equals
        \[
        \begin{cases}
			\min\{{p^{p^k+v(t_1,_{n-1}t_2)},p^{2p^k}}\} &  \text{if } \nu(t_1)\le\nu(t_2), \mathrm{deg}(g)\geq (n-1)p^{\nu(t_2)}+p^{\nu(t_1)}+1,\\		
			\min\{{p^{p^k+v(t_2,_{n-1}t_2)},p^{2p^k}}\} & \text{if } \nu(t_1)>\nu(t_2),
			\mathrm{deg}(g)\geq np^{\nu(t_2)}+1,\\
			0 & \text{otherwise}.
		\end{cases} \]

	Now, looking at the formula for $s_{t_1,t_2}(n,g)$ given by \cref{lemma:countingSg}, it is immediate to see that $s_{t_1,t_2}(n,1)\neq 0$ for all $t_1,t_2\in T$. This is not the case when $g\neq 1$. For example, as already mentioned, we have $s_{t_1,t_2}(n,g)=0$ whenever $\mathrm{deg}(g)\leq n$. The following three lemmas treat some special cases when $\mathrm{deg}(g)\geq n+1$ for $n$ a $p$-power. Recall that $l(g)=\lceil \log_p(\deg(g))\rceil$ for $g\in G$. 
    
	\begin{lemma}\label{zeros1}
		Let $g\in W_k$ such that $l:=l(g)\geq 1$. Then $s_{t_1,t_2}(1,g)=0$ if and only if $\mathrm{min}\{\nu(t_1),\nu(t_2)\}\geq l$ and $g\neq1$.
	\end{lemma}
	\begin{proof}
		Suppose that $\nu(t_1)\le\nu(t_2)$. Then, by \cref{lemma:countingSg}, $s_{t_1,t_2}(1,g)=0$ if and only if $\mathrm{deg}(g)\leq p^{\nu(t_1)}$. Since $\nu(t_1)$ is an integer, the previous inequality is satisfied if and only if \[\mathrm{min}\{\nu(t_1),\nu(t_2)\}=\nu(t_1)\geq \lceil\mathrm{log}_p(\mathrm{deg}(g))\rceil=l. \]
		The case $\nu(t_1)> \nu(t_2)$ is analogous.
	\end{proof}

    \begin{lemma}\label{zeros}
		Suppose that $n=p^m$ for some $1\leq m\leq k-1$. Let $g\in W_k$ such that $l:=l(g)\geq m+1$. Then
		$s_{t_1,t_2}(n,g)=0$ for every $t_1,t_2$ such that one of the following holds:
		\begin{enumerate}
			\item[(i)] $\nu(t_2)\geq l-m+1$ and $g\neq 1$;
			\item[(ii)] $\nu(t_2)=l-m$ and $\nu(t_1)\geq l-m$ and $g\neq 1$.
		\end{enumerate}
		Moreover, $s_{t_1,t_2}(n,g)\neq 0$ for every $t_1,t_2$ such that $\nu(t_2)\leq l-m-1$.
	\end{lemma}
	\begin{proof}
		Suppose that $\nu(t_2)\geq l-m+1$. Then,
        \[
        (n-1)p^{\nu(t_2)}
        \geq (n-1)p^{l-m+1}
        =\left(p\cdot\frac{n-1}{n}\right)p^l
        \geq \left(p\cdot\frac{n-1}{n}\right)\mathrm{deg}(g).
        \] 
		Since $n\geq 2$, we have $\frac{n-1}{n}\geq 1/2$, and hence $(p\cdot\frac{n-1}{n})\geq p/2\geq 1$. Therefore, $\mathrm{deg}(g)\leq (n-1)p^{\nu(t_2)}$ and the first item follows by \cref{lemma:countingSg}. 
	    If $\nu(t_2)=l-m$, note that $np^{\nu(t_2)}=p^l\geq \deg(g)$. Again, by \cref{lemma:countingSg}, we have $s_{t_1,t_2}(n,g)=0$ whenever $\nu(t_1)\geq l-m$.

        Now, suppose $\nu(t_2)\leq l-m-1$. If $\nu(t_1)\leq \nu(t_2)$, then
        \[
        (n-1)p^{\nu(t_2)}+p^{\nu(t_1)}
        \le \frac{n-1}{n}\cdot p^{l-1}+\frac{1}{n}\cdot p^{l-1}=p^{l-1}=p^{\varepsilon(g)-1}\mathrm{deg}(g)
        \]
        where $\varepsilon(g)=l - \mathrm{log}_p(\deg(g))$. Since $\varepsilon(g)\in[0,1)$, we have $p^{\varepsilon(g)-1}<1$, and hence $(n-1)p^{\nu(t_2)}+p^{\nu(t_1)}<\mathrm{deg}(g)$. Similarly, if $\nu(t_1)> \nu(t_2)$, then $np^{\nu(t_2)}=p^{l-1}<\deg(g)$. Therefore, the lemma follows.
	\end{proof}
		
	\begin{remark}
    \label{rem:2cases}
	    Suppose that $n\geq 2$, $\nu(t_2)=l-m$ and $i=\nu(t_1)<l-m$. In what follows, we show that $s_{t_1,t_2}(n,g)$ can be zero or non-zero depending on $g$.
        %by the previous lemma are those where $\nu(t_2)=l-m$ and $i=\nu(t_1)<l-m$. Although a more detailed analysis of these cases goes beyond our purposes, we illustrate two particular cases below. 
        Recall that $\varepsilon(g)=\lceil \mathrm{log}_p(\mathrm{deg}(g))\rceil - \mathrm{log}_p(\deg(g))\in [0,1)$ and %; we consider the extremal cases $\varepsilon(g)\geq m-\mathrm{log}_p(n-1)$ and $\varepsilon(g)=0$. 
        note that \[ (n-1)^{l-m}=\biggl(p^{\varepsilon(g)}\cdot\frac{n-1}{n}\biggr)\mathrm{deg}(g). \]
	If $\varepsilon(g)\geq m-\mathrm{log}_p(n-1)$, it follows that $p^{\varepsilon(g)}\cdot\frac{n-1}{n}\geq 1$, and hence \[ \biggl(p^{\varepsilon(g)}\cdot\frac{n-1}{n}\biggr)\mathrm{deg}(g) +p^i > \mathrm{deg}(g)  \] for every $i=0,\dots,l-m-1$. In this case, $s_{1,p^{l-m}}(n,g),\cdots,s_{p^{l-m-1},p^{l-m}}(n,g)=0$. On the other hand, if $\varepsilon(g)=0$, that is, if $\mathrm{deg}(g)$ is a $p$-power, it follows \[ \biggl(p^{\varepsilon(g)}\cdot\frac{n-1}{n}\biggr)\mathrm{deg}(g) +p^i<\biggl(\frac{n-1}{n}+\frac{1}{n}\biggr)\mathrm{deg}(g)=\mathrm{deg}(g). \] Therefore, $s_{1,p^{l-m}}(n,g),\cdots,s_{p^{l-m-1},p^{l-m}}(n,g)\neq0$ for $\varepsilon(g)=0$. 
	\end{remark} 

    \begin{lemma}\label{lemma:s_identity}%\comtommaso{Instead of $i,j$      use $\nu(t_1),\nu(t_2)$ to include the cases $t_1=0$ or $t_2=0$.    }
       Suppose that $n=p^m$ for some $1\leq m\leq k-1$. Let $t_1,t_2\in T$ and suppose that either:
       \begin{enumerate}
          \item[(i)] $\nu(t_2)\ge k-m+1$ or 
          \item[(ii)] $\nu(t_2)=k-m$ and $\nu(t_1)\ge k-m$.
       \end{enumerate}
       Then we have that $s_{t_1,t_2}(n,1)=p^{2p^k}$.%\comtommaso{$p^{2p^k}$}
    \end{lemma}
    \begin{proof}
        Set $\nu(t_1)=i$, $\nu(t_2)=j$ and recall that $\nu(0)=\infty$. By definition we need to calculate
        \[
        s_{t_1,t_2}(n,1) = \begin{cases}
        \min\{ p^{p^k+(n-1)p^j+p^i},p^{2p^k} \} & \text{if } i\le j \\
        \min\{ p^{p^k+n p^j},p^{2p^k} \} & \text{if }  i > j.
        \end{cases}
        \]
        If $j\ge k-m+1$ and $i\le j$, we have $$(n-1)p^j+p^i = p^{j+m}-p^j+p^i\ge p^{k+1} -p^j + p^i\ge p^{k+1}-p^{k}+1 = p^k(p-1)+1,$$
        which is greater than $p^k$, so the minimum is $p^{2p^k}$.

        On the other hand, if  $j\ge k-m+1$ and $i< j$, we have 
        \[
        np^j \ge p^m p^{k-m+1} = p^{k+1}
        \]
        so again the minimum is $p^{2p^k}$. The case $j=k-m$ and $i\ge k-m$ is handled similarly. 
    \end{proof}

    We now proceed with the computation of the upper and lower bounds for the sum in E\-qua\-tion (\ref{eq:sum_g}) when $n$ is a $p$-power.
    By \cref{lemma:countingSg}, $s_{t_1,t_2}(n,g)$ depends only on $\nu(t_1)$ and $\nu(t_2)$, so that it will be useful to break the sum in \cref{eq:sum_g} according to the $p$-valuations of $t_1,t_2$. For every $j=0,\dots,k-1$, let \[ T_j=\{ t\in T\mid \nu(t)=j \}\quad \text{and}\quad T_{\geq j}=\{t\in T\mid \nu(t)\geq j\}.\]
	In the following lemma, we count how many $t\in T$ have the same $p$-valuation. 
	
	\begin{lemma}\label{lemma:euler}
		In the notation introduced above, $T$ is the disjoint union \[ T= \{0\} \sqcup %\biggl(
		\bigsqcup_{j=0}^{k-1}T_j
		%\biggr)
		\qquad \text{and} \qquad |T_j|=\biggl(1-\frac{1}{p}\biggr)p^{k-j} \] for every $j=0,\dots,k-1$. In particular,  $|T_{\geq j}|=p^{k-j}$.
	\end{lemma}
	\begin{proof}
		For every $m\in\N_0$, denote $\mathcal{T}_m=\{ i\mid 1\leq i\leq m, \mathrm{gcd}(i,m)=1 \}$, so that $|\mathcal{T}_m|=\vphi(m)$, where $\vphi$ is the Euler's totient function. For every $j=0,\dots,k-1$, let us define
        \begin{align*}
			f_j\colon T_j&\to \mathcal{T}_{p^{k-j}}\\
			t&\mapsto t/p^j.
		\end{align*}
		Since every $t\in T_j$ can be written as $t=p^ji$ with $\mathrm{gcd}(i,p)=1$ and $1\leq i\leq p^{k-j}$, the map $f_j$ is a bijection. Therefore, \[ |T_j|=|\mathcal{T}_{p^{k-j}}|=p^{k-j}-p^{k-j-1}=\biggl(1-\frac{1}{p}\biggr)p^{k-j} \] for every $j=0,\dots,k-1$. Now, note that $T_{\geq j}=T_j\sqcup\cdots\sqcup T_{k-1}\sqcup \{0\}$. Then \[ |T_{\geq j}|=\sum_{m=j}^{k-1}|T_m|+|\{0\}|=\sum_{m=j}^{k-1}(p^{k-m}-p^{k-m-1})+1=p^{k-j}. \qedhere\]
	\end{proof}
        
	Since we will make many estimates, we present four lemmas in advance. The first two involve some arithmetic functions that will appear as certain exponents, while in the third and fourth we compute some estimates of the parts into which we will break the sum according to the $p$-valuations of $t_1$ and $t_2$. 
    \begin{lemma}\label{lemma:arith}
		Let $f_1,f_2\colon\N_0^2\to\Z$ be the functions defined by \[ f_1(i,j)=(n-1)p^j+p^i-i,\qquad f_2(i,j)=np^j-i. \] Let $f(j)=f_2(j,j)=np^j-j$. Then the following holds:
        \begin{enumerate}
            \item[(i)] If $n=1$, then $f_1(i,j)=f(i)$ for every $i,j$ and $f_2(i,j)\leq f(j)$ whenever $i\geq j$;
            \item[(ii)] If $n\geq 2$, then $f_1(i,j)\leq f(j)$ whenever $i\leq j$ and $f_2(i,j)\leq f(j)$ whenever $i\geq j$. 
        \end{enumerate} 
	\end{lemma}
	\begin{proof}
		%If $i\geq j$, then \[ f_2(i,j)=np^j-i\leq np^j-j=f(j). \] 
    Assume that $n\geq 2$. Note that $p^i-i$ is increasing if $i\geq 0$.  If $i\leq j$, then \[ f_1(i,j)=(n-1)p^j+(p^i-i)\leq (n-1)p^j+(p^j-j)= np^j-j=f(j).\] 
    The other cases follow similarly.
	\end{proof}

    \begin{lemma}\label{lem:hyperexp}
        Let $M,N\in\N_0$. Then the following holds:
\[            
 %\[\sum_{i=0}^Np^{p^{i}-i}\leq 2p^{p^N-N};\]
           \mathrm{(i)}\ \ \sum_{i=0}^Np^{p^{i+M}-i}\leq 2p^{p^{N+M}-N}; \qquad
            \mathrm{(ii)}\ \ %\[\sum_{i=0}^Np^{p^i-2i}\leq 4p^{p^N-2N}.\]
              \sum_{i=0}^Np^{p^{i+M}-2i}\leq 4p^{p^{N+M}-2N}. 
\]
        %Moreover, for every $M\in\N_0$, we have
        %\[ \sum_{i=0}^Np^{p^{i+M}-i}\leq 2p^{p^{N+M}-N}\quad\text{and}\quad \sum_{i=0}^Np^{p^{i+M}-2i}\leq 4p^{p^{N+M}-2N}. \]
    \end{lemma}
    % \begin{proof}
    %     We first use induction on $N$ to show that $\sum_{i=0}^Np^{p^{i}-i}\leq 2p^{p^N-N}$. The cases $N=0$ and $N=1$ are easy to check. If $N\geq 1$, by inductive hypothesis we have
    %     \[ \sum_{i=0}^{N+1}p^{p^i-i}\leq 2p^{p^N-N}+p^{p^{N+1}-(N+1)}. \]
    %     Then, we need to show that $p^{p^{N+1}-(N+1)}\geq 2p^{p^N-N}$. Taking the logarithms in base $p$, the previous inequality is equivalent to 
    %     \[ p^N\geq \frac{\log_p(2)+1}{p-1}, \]
    %     which is true since $p^N\geq p\geq\frac{\log_p(2)+1}{p-1}$. Therefore,
    %     \[\sum_{i=0}^Np^{p^{i+M}-i}\leq p^M\sum_{i=0}^{M+N}p^{p^i-i}\leq 2p^Mp^{p^{M+N}-(N+M)}=2p^{p^{N+M}-N},\]
    %     which proves (i).

    %     To prove (ii) %\comment{Solo dimostrerei (ii) e direi che (i) è più facile}
    %     we argue in a similar way. First, we prove by induction on $N$ that $\sum_{i=0}^Np^{p^i-2i}\leq 4p^{p^N-2N}$. As before, the cases $N=0$, $N=1$ and $N=2$ can be checked by hand. If $N\geq 2$, by inductive hypothesis we have 
    %     \[ \sum_{i=0}^{N+1}p^{p^i-2i}\leq 4p^{p^N-2N}+p^{p^{N+1}-2(N+1)}, \] and hence we need to show that $p^{p^{N+1}-2(N+1)}\geq \frac{4}{3}p^{p^N-2N}$. As in the first item, this is equivalent to
    %     \[ p^N\geq \frac{\log_p(\frac{4}{3})+2}{p-1}, \]
    %     which is true since $p^N\geq p^2\geq\frac{\log_p(\frac{4}{3})+2}{p-1}$. Therefore, 
    %     \[ \sum_{i=0}^Np^{p^{i+M}-2i}\leq p^{2M}\sum_{i=0}^{M+N}p^{p^i-2i}\leq 4p^{2M}p^{p^{M+N}-2(N+M)}=4p^{p^{N+M}-2N}.\qedhere \]
    % \end{proof}
    \begin{proof}
        We will prove (i), respectively (ii), by induction on $N$. The cases $N=0$ and $N=1$, respectively $N=0$, $N=1$, and $N=2$, are easy to check by hand. If $N\geq 1$, respectively $N\ge 2$, by inductive hypothesis we have
        \[ \sum_{i=0}^{N+1}p^{p^{i+M}-i}\leq 2p^{p^{N+M}-N}+p^{p^{N+M+1}-(N+1)}, \]
        respectively
        \[ \sum_{i=0}^{N+1}p^{p^{i+M}-2i}\leq 4p^{p^{N+M}-2N}+p^{p^{N+M+1}-2(N+1)}.\]
        Therefore, we only need to check that $p^{p^{N+M+1}-(N+1)}\geq 2p^{p^{N+M}-N}$, respectively $p^{p^{N+M+1}-2(N+1)}\geq \frac{4}{3}p^{p^{N+M}-2N}$. Taking logarithms in base $p$ in both sides, the previous inequality is equivalent to $p^{N+M}\geq (\log_p(2)+1)/(p-1),$ respectively $p^{N+M}\geq (\log_p(\frac{4}{3})+2)/(p-1),$
        which is true since $p^{N+M}\geq p\geq (\log_p(2)+1)/(p-1)$, respectively $p^{N+M}\geq p^2\geq(\log_p(\frac{4}{3})+2)/(p-1)$.
    \end{proof}
    %{\color{violet} \begin{remark}
    %    Let $\beta>0$ such that $\sum_{i=0}^Np^{p^i-2i}\leq (1+\beta)p^{p^N-2N}$. Take $N=2$ and consider \[ p+p^{p-2}+p^{p^2-4}\leq (1+\beta)p^{p^2-4} \] for some $\beta>0$. The previous is equivalent to \[ p+p^{p-2}\leq \beta p^{p^2-4}. \] If $p=2$, then we have $2+1\leq \beta$, that is, $\beta \geq 3$.
    %\end{remark} }
    %In order to estimate finite sums involving exponentials, we will also use the following elementary observation. Let $N\geq 0$ and let $b_i\in\mathbb{R}$ for $i=0,\dots, N$. Let $\bar{i}\in\mathrm{argmax}\{b_i\mid i=0,\dots,N\}$, that is, $b_i\leq b_{\bar{i}}$ for every $i=0,\dots, N$. For every $a>0$, we have
%\[ a^{b_{\bar{i}}} \leq \sum_{i=0}^Na^{b_i} \leq (N+1)a^{b_{\bar{i}}}. \]
    
	In what follows, let 
	\[ S_0(n,g) :=\sum_{t_1\in T}s_{t_1,0}(n,g)+\sum_{t_2\in T, t_2\neq 0}s_{0,t_2}(n,g) \]
	and
	\[ S_1(n,g) := \sum_{t_1,t_2\in T}s_{t_1,t_2}(n,g)- S_0(n,g). \]
 In the following lemmas we will use \cref{lemma:arith} without explicitly mentioning it. Recall that $\delta_1(g)$ is the characteristic function of $1$: $\delta_1(1)=1$ and $\delta_1(g)=0$ for $1\neq g\in G$; we also use the convention $l(1)=k$.

\begin{lemma}\label{lemma:boundsS0}
		Suppose that $n=p^m$ for some $0\leq m\leq k-1$. Let $g\in W_k$ be such that $l=l(g)\ge m+1$ (if $g=1$, take $l=k$). Then 
%       {\color{blue}
%        $$
%		\frac{\delta_1(g)}{p^{(1+\delta_0(m))k}}+\frac{c_0(n,p)p^{p^{l-1}-l}}{p^{p^k+k}}\leq \frac{S_0(n,g)}{p^{2p^k+2k}}\leq \frac{2\delta_1(g)}{p^{(1+\delta_0(m))k}}+\frac{c_0(n,p)(l-m)p^{p^{l-1}-l}}{p^{p^k+k}},
%        $$
%        where $c_0(n,p)=n(1+\delta_0(m))(p-1)$.
%        }
        \begin{align*}
		\frac{\delta_1(g)}{p^{(1+\delta_0(m))k}}+ n(p-1)        (1+\delta_0(m)) \frac{p^{p^{l-1}-l}}{p^{p^k+k}}&\leq \frac{S_0(n,g)}{p^{2p^k+2k}} \\
        &\le
        %(2-\delta_0(m)) 
        \frac{2\delta_1(g)}{p^{(1+\delta_0(m))k}}+ 4 n(p-1) \frac{p^{p^{l-1}-l}}{p^{p^k+k}}.
        \end{align*}
      %  \begin{multline*}
		%\frac{\delta_1(g)}{p^{(1+\delta_0(m))k}}+ n(p-1) \frac{p^{p^{l-1}-l}}{p^{p^k+k}}\leq 
       % \frac{S_0(n,g)}{p^{2p^k+2k}}\leq \frac{2\delta_1(g)}{p^{(1+\delta_0(m))k}}+ {\color{violet}2(1+\delta_0(m))}n(p-1) \frac{p^{p^{l-1}-l}}{p^{p^k+k}}.
       % \end{multline*}
        %\commatteo{the first one is with more precise constants, the second more concise}
       % \comtommaso{I have written in violet $2(1+\delta_0(m))$ instead of $4$}
	\end{lemma}
    \begin{proof}
        Note that $s_{t,0}(1,g)=s_{0,t}(1,g)$ for every $t\in T$. By \cref{zeros1}, $s_{t,0}(1,g)=0$ if and only if $\nu(t)\geq l$ and $g\neq 1$. Therefore, by \cref{lemma:countingSg} and \cref{lemma:euler} we have
        \begin{equation}\label{eq:S01}
            S_0(1,g) = \delta_1(g)s_{0,0}(1,1) + 2\sum_{i=0}^{l-1}|T_i|s_{p^i,0}(1,g)
        = \delta_1(g)p^{2p^k} + 2\biggl(1-\frac{1}{p}\biggr)p^{p^k+k}\sum_{i=0}^{l-1}p^{f(i)}
        \end{equation}
		  where $f(i)=p^i-i$. Since $f(i)\leq f(l-1)=p^{l-1}-l+1$ for every $i=0,\dots,l-1$, using \cref{lem:hyperexp}, it follows that
		\[
        2(p-1)p^{p^{l-1}-l}
        \le 2\biggl(1-\frac{1}{p}\biggr)\sum_{i=0}^{l-1}p^{f(i)} \le  4 \left(p-1\right) p^{p^{l-1}-l}.
        \]
        Note that the lower bound is obtained by considering the term corresponding to $i=l-1$ in \cref{eq:S01}. Therefore, we obtain
        \[
         \frac{\delta_1(g)}{p^{2k}}+\frac{2(p-1)p^{p^{l-1}-l}}{p^{p^k+k}}
        \le \frac{S_0(1,g)}{p^{2p^k+2k}}
        \le
        \frac{\delta_1(g)}{p^{2k}}+\frac{4(p-1)p^{p^{l-1}-l}}{p^{p^k+k}}.
        \]
        Now, assume $n\geq 2$. By \cref{zeros}, $s_{t_1,0}(1,g)=0$ and $s_{0,t_2}(1,g)=0$ if $t_1\in T$, $t_2\in T_{\geq l-m}$ and $g\neq 1$. Thus, by \cref{lemma:countingSg}, \cref{lemma:s_identity} and \cref{lemma:euler} we have
        \begin{align}\label{eq:S0n}  
        S_0(n,g)
        &= \delta_1(g)\biggl(\sum_{t_1\in T}s_{t_1,0}(n,1)+\sum_{t_2\in T_{\geq k-m}\smallsetminus\{0\}}s_{0,t_2}(n,1)\biggr)+\sum_{j=0}^{l-m-1}|T_j|s_{0,p^j}(n,g)\notag\\
        %\delta_1(g)(|T|+|T_{\geq k-m}|-1)p^{2p^k}+\sum_{j=0}^{l-m-1}|T_j|s_{0,p^j}(n,g)\\
        &=\delta_1(g)p^{2p^k}(p^{k}+p^{m}-1)+\biggl(1-\frac{1}{p}\biggr)p^{p^k+k}\sum_{j=0}^{l-m-1}p^{f(j)},
        \end{align}
		  where $f(j)=np^j-j$. Since $f(j)\le f(l-m-1)=p^{l-1}-l+m+1$, we have
        \[
        n(p-1)p^{p^{l-1}-l}\le
        \biggl(1-\frac{1}{p}\biggr)\sum_{j=0}^{l-m-1}p^{f(j)}\le 2n(p-1) p^{p^{l-1}-l}.
        \]
        Note that the lower bound is obtained by considering the term corresponding to $j=l-m-1$ in \cref{eq:S0n} and that we have used \cref{lem:hyperexp} in the upper bound. 
        Therefore, 
         \begin{align*}
        \frac{\delta_1(g)(1+p^{m-k}-p^{-k})}{p^{k}}
        &+
        \frac{n(p-1)p^{p^{l-1}-l}}{p^{p^k+k}}
        \le \frac{S_0(n,g)}{p^{2p^k+2k}}\\
        &\le
        \frac{\delta_1(g)(1+p^{m-k}-p^{-k})}{p^{k}}
        +
        \frac{2n(p-1)p^{p^{l-1}-l}}{p^{p^k+k}}. \qedhere
        \end{align*}
    \end{proof}

\begin{lemma}\label{lemma:boundsS1}
		Suppose that $n=p^m$ for some $0\leq m\leq k-1$. Let $g\in W_k$ be such that $l=l(g)\ge m+1$ (if $g=1$, take $l=k$). Then,
        \begin{multline*}
           \left(1+\delta_0(m)\frac{2}{p}\right)n^2 (p-1)^2 \frac{p^{p^{l-1}-2l}}{p^{p^k}} + \delta_1(g) h(k,m) \le \frac{S_1(n,g)}{p^{2p^k+2k}}\\
           \le n^2 (p-1)^2 \frac{(8k+4) p^{p^{l-\delta_0(m)}-2l}}{p^{p^k}}  + \delta_1(g) h(k,m)
    \end{multline*}

where $$ h(k,m)=  \frac{ (p-1) p^{m-1}(p^m-1) + (1-\delta_0(m)) (p^{m-1}-1)(p^k-1)}{p^{2k}}.$$
\end{lemma}
    \begin{proof}
        Let $t_1,t_2\in T\smallsetminus\{0\}$. Then, by \cref{zeros1}, we have $s_{t_1,t_2}(1,g)=0$ if and only if $\min\{\nu(t_1),\nu(t_2)\}\ge l$ and $g\neq 1$, and hence, by \cref{lemma:countingSg},
        \begin{align*}
            &S_1(1,g)=
            \sum_{j=0}^{l-1}|T_j|\biggl(\sum_{i=j}^{k-1}|T_i|s_{p^i,p^j}(1,g)\biggr)
            +
            \sum_{i=0}^{l-1}|T_i|\biggl(\sum_{j=i+1}^{k-1}|T_j|s_{p^i,p^j}(1,g)\biggr) \\[1em]
            &=\,p^{p^k+2k} \left(1-\frac{1}{p}\right)^{\hspace{-2pt}2} \cdot \sum_{i=0}^{l-1}p^{p^i-2i} + 2p^{p^k+2k} \left(1-\frac{1}{p}\right)^{\hspace{-2pt}2} \cdot \sum_{j=0}^{l-1}p^{-j}\biggl(\sum_{i=j+1}^{k-1}p^{f_2(i,j)}\biggr),
        \end{align*}
        where $f_2(i,j)=p^j-i$.
        Similar to the proof of \cref{lemma:boundsS0}, by \cref{lem:hyperexp}, we have
        \[
        (p-1)^2 p^{p^{l-1}-2l}
        \le
        \left(1-\frac{1}{p}\right)^{\hspace{-2pt}2} \cdot \sum_{i=0}^{l-1}p^{{p^i}-2i} \le 4(p-1)^2 p^{p^{l-1}-2l}.
        \]
        
        Similarly, we have $f_2(i,j)\le f_2(l,l-1)=p^{l-1}-l$ for all $ j=0,\ldots,l-1$, $i=j+1,\ldots, k-1$, so we obtain
        \begin{align*}
            2(p-1)^2p^{p^{l-1}-2l-1}
            &\le
            2\left(1-\frac{1}{p}\right)^{\hspace{-2pt}2} \cdot \sum_{j=0}^{l-1}p^{-j}\biggl(\sum_{i=j+1}^{k-1}p^{f_2(i,j)}\biggr)\\
            &=  2\left(1-\frac{1}{p}\right)^{\hspace{-2pt}2} \cdot \sum_{j=0}^{l-1} p^{-j} \sum_{i=j+1}^{k-1} p^{p^{j}-i}
            \\
            & \le 2\left(1-\frac{1}{p}\right)^{\hspace{-2pt}2} \cdot \sum_{j=0}^{l-1} \sum_{i=j+1}^{k-1} p^{p^{j}-2j}\\
            &\le 8k\left(1-\frac{1}{p}\right)^{\hspace{-2pt}2} p^{p^{l-1}-2(l-1)} 
            = 8 k(p-1)^{2} p^{p^{l-1}-2l}
        \end{align*}
        where for the upper bound we have also used that $k-j-1\le k$ for $j=0,\dots, l-1$. Therefore,
\begin{equation*}
            \left(1+\frac{2}{p}\right)\left(p-1\right)^2 \frac{ p^{p^{l-1}-2l}}{p^{p^k}}
            \le
            \frac{S_1(1,g)}{p^{2p^k+2k}}
            \le
            (8k+4) (p-1)^2  \frac{p^{p^{l-1}-2l}}{p^{p^k}}.
\end{equation*}

        For $n\ge 2$, by \cref{zeros} we have
        \begin{multline}
%        \begin{split}
             S_1(n,g)
            =
            \sum_{j=0}^{l-m-1}|T_j|\sum_{i=0}^{k-1}|T_i|s_{p^i,p^j}(n,g)
            +
            \lvert T_{l-m}\rvert\sum_{i=0}^{l-m-1}|T_i|s_{p^i,p^{l-m}}(n,g)\\[1em]
            +
            \delta_1(g)\biggl(|T_{k-m}|\sum_{i=k-m}^{k-1}|T_i|s_{p^i,p^{k-m}}(n,1)
            +
            \sum_{j=k-m+1}^{k-1}\hspace{-2pt}\lvert T_j\rvert \sum_{i=0}^{k-1}|T_i|s_{p^i,p^j}(n,1)\biggr).
%                \end{split}
            \label{eq:empty}
        \end{multline}
        
         We will deal with upper and lower bounds of each of the first and second summands in \cref{eq:empty} one by one, commenting on the choices that we will make to obtain the displayed lower and upper bounds. Then we will calculate exactly the third and fourth summands.
        
        For the lower bound of the first summand in \cref{eq:empty} we choose $i=j=l-m-1$. Thus,
\begin{multline*}
   \sum_{j=0}^{l-m-1}|T_j| \sum_{i=0}^{k-1}|T_i|s_{p^i,p^j}(n,g)  \ge \left( 1-\frac{1}{p} \right)^2 p^{2(k-l+m+1)} p^{p^k+n p^{p^{l-m-1}}} \\
   \ge n^2 (p-1)^2 p^{2k+p^k} p^{p^{l-1}-2l}.
\end{multline*}
%{\color{red} 
%\[
% (p-1)^2 p^{p^k + p^{l-1} + 2k - 2l + 2m} < S \le 4(p-1)^2 p^{p^k + p^{l-1} + 2k - 2l + 2m}
%\]
%}
For the upper bound of the first summand in \cref{eq:empty}, we use Lemma~\ref{lemma:arith} and \cref{lem:hyperexp} to obtain
\begin{align*}
   \sum_{j=0}^{l-m-1}|T_j| &\sum_{i=0}^{l-m-1}|T_i|s_{p^i,p^j}(n,g) \\
   &= \left(p-1\right)^2 p^{p^k+2k} \sum_{j=0}^{l-m-1} \biggl( \sum_{i=0}^j p^{(n-1)p^j+p^i-i-j} + \sum_{i=j+1}^{l-m-1} p^{n p^j-i-j}   \biggr)\\
   &\le \left(p-1\right)^2 p^{p^k+2k} \sum_{j=0}^{l-m-1} \biggl(  \sum_{i=0}^j p^{n p^j-2j} + \sum_{i=j+1}^{l-m-1} p^{n p^j-2j}   \biggr)\\
   &= \left(p-1\right)^2 p^{p^k+2k}  \sum_{i,j=0}^{l-m-1} p^{n p^j-2j} \le 4k n^2  \left(p-1\right)^2 p^{p^k+2k} p^{p^{l-1}-2l}.
\end{align*}

We will now deal with the second term in \cref{eq:empty}. First of all, we observe that
 the terms $s_{p^i,p^j}(n,g)$ are not easy to describe when $j=l-m$, $i\le l-m-1$ and $g\neq 1$ (see \cref{rem:2cases}). Thus, in order to compute the lower and upper bounds of $S_1(n,g)$, we will use that $0\le |T_i|s_{p^i,p^{l-m}}(n,g)\le p^{p^k+k+f_2(i,l-m)}$ for any $i\le l-m-1$. Hence, we take zero\footnote{We note that if $\deg(g)$ is a $p$-power, we could use \cref{rem:2cases} to obtain a lower bound of order $p^{2k+p^k} p^{p^l-p^{l-m}+p^{l-m-1}-2l+1}$. We decided to include this for simplicity and because all other results in this article remain true without it.} as lower bound for the second summand in \cref{eq:empty}. On the other hand, 
%\begin{multline*}
%     \lvert T_{l-m}\rvert \sum_{i=0}^{l-m-1}\lvert T_i\rvert s_{p^i,p^{l-m}}(n,g) \ge {\color{red}\delta_0(\varepsilon(g))} n \left(1-\frac{1}{p} \right)^2 p^{p^l-p^{l-m}-l} p^{2k+p^k} p^{p^{l-m-1}-l+m+1} = \\ = {\color{red}\delta_0(\varepsilon(g))} n^2 \left(1-\frac{1}{p} \right)^2 p^{2k+p^k} p^{p^l-p^{l-m}+p^{l-m-1}-2l+1} 
%\end{multline*}
for the upper bound for the second summand in \cref{eq:empty} we have
\begin{align*}
  |T_{l-m}| &\sum_{i=0}^{l-m-1}|T_i|s_{p^i,p^{l-m}}(n,g) \le \left(1-\frac{1}{p}\right)^2 n p^{p^k+2k-l} \sum_{i=0}^{l-m-1} p^{(n-1)p^{l-m}+p^i-i}\\
  &=
  \left(1-\frac{1}{p}\right)^2 n p^{p^k+2k}  p^{(1-\frac{1}{n})p^l-l} \sum_{i=0}^{l-m-1} p^{p^i-i}\\
  &= 2(p-1)^2n^2p^{p^k+2k}p^{(1-\frac{1}{n})p^l+\frac{1}{n}p^{l-1}-2l-1}\\
  &\leq 2(p-1)^2n^2p^{p^k+2k}p^{p^l-2l-1}.
\end{align*}

For the third summand in \cref{eq:empty}, we use \cref{lemma:s_identity}, so that the sum becomes 
\[
  \left(1-\frac{1}{p}\right)^2 p^m p^{2p^k} \sum_{i=k-m}^{k-1} p^{k-i} = \frac{(p-1)^2}{p^2} p^m p^{2p^k} p \frac{p^m-1}{p-1} = (p-1) p^{m-1}(p^m-1) p^{2p^k}.
\]

Finally, again using \cref{lemma:s_identity}, for the fourth summand in \cref{eq:empty} we have that
\begin{align*}
   \left(1-\frac{1}{p}\right)^2 p^{2p^k} &\sum_{j=k-m+1}^{k-1} p^{k-j} \sum_{i=0}^{k-1} p^{k-i} \\
   &= (1-\delta_{0}(m)) \frac{(p-1)^2}{p^2} p^{2p^k} p \frac{p^{m-1}-1}{p-1} p \frac{p^{k}-1}{p-1} \\
   &= (1-\delta_{0}(m)) (p^{m-1}-1)(p^k-1) p^{2p^k}.
\end{align*}

Collecting all of our lower and upper bounds for $n\ge 2$, we can conclude that

 \begin{multline*}
 n^2  \left(p-1\right)^2  \frac{ p^{p^{l-1}-2l}}{p^{p^k}}  + 
 \delta_1(g) h(k,m) \le \frac{S_1(n,g)}{p^{2p^k+2k}} \\[1em]
 %  \le  n^2  \left(p-1\right)^2  \frac{4k p^{p^{l-1}-2l} + (p^{-m}-p^{-l}) p^{p^l-l-2}}{p^{p^k}}   + \delta_1(g) h(k,m) \\[1em]
 \le  n^2  \left(p-1\right)^2  \frac{4k p^{p^{l-1}-2l} + 2 p^{p^l-2l-1}}{p^{p^k}}   + \delta_1(g) h(k,m),
\end{multline*}
    where $$ h(k,m)=  \frac{ (p-1) p^{m-1}(p^m-1) + (1-\delta_0(m)) (p^{m-1}-1)(p^k-1)}{p^{2k}}, $$
and thus for any $n\ge 1$ we have that 
\begin{align*}
   n^2 (p-1)^2 \frac{p^{p^{l-1}-2l}}{p^{p^k}} + \delta_1(g) &h(k,m) \le \frac{S_1(n,g)}{p^{2p^k+2k}} \\
   %\le n^2 (p-1)^2 \frac{(8k+4) p^{p^{l}-l}}{p^{p^k}}  + \delta_1(g) h(k,m) \\
   &\le n^2 (p-1)^2 \frac{(8k+4) p^{p^{l-\delta_0(m)}-2l}}{p^{p^k}}  + \delta_1(g) h(k,m).\qedhere
\end{align*}
\end{proof}

\begin{remark}
    Let us briefly comment on the asymptotic behaviour of the quan\-ti\-ties involved in \cref{lemma:boundsS0} and \cref{lemma:boundsS1}. There are two different regimes depending on whether $g=1$ or not. If $g=1$, these quantities tend to zero at least as fast as $p^{-k}$. On the other hand, if $g\neq 1$, then these quantities tend to zero at least as fast as $p^{-p^k}$.
\end{remark}

	We are now ready to prove \cref{thm:engelbound_gen}.
        
	\begin{proof}[Proof of \cref{thm:engelbound_gen}]
			
		Recall that \[ P(W_k,e_n,g)=\frac{1}{|W_k|^2}\sum_{t_1,t_2\in T} s_{t_1,t_2}(n,g)=\frac{S_0(n,g)+S_1(n,g)}{p^{2p^k+2k}} \] 
        for any $g\in W_k$.
        By \cref{lemma:boundsS0} and \cref{lemma:boundsS1}, we have
  %      \begin{multline*}
  %		\frac{\delta_1(g)}{p^{(1+\delta_0(m))k}}+ n(p-1)        {\color{red}(1+\delta_0(m))} \frac{p^{p^{l-1}-l}}{p^{p^k+k}}\leq \frac{S_0(n,g)}{p^{2p^k+2k}}\leq \\ %(2-\delta_0(m)) 
  %      \frac{2\delta_1(g)}{p^{(1+\delta_0(m))k}}+ 4 n(p-1) \frac{p^{p^{l-1}-l}}{p^{p^k+k}}.
  %      \end{multline*}
  %      \begin{multline*}
  %         \left(1+\delta_0(m)\frac{2}{p}\right)n^2 (p-1)^2 \frac{p^{p^{l-1}-2l}}{p^{p^k}} + \delta_1(g) h(k,m) \le \frac{S_1(n,g)}{p^{2p^k+2k}} \le \\ \le n^2 (p-1)^2 \frac{(8k+4) p^{p^{l-\delta_0(m)}-2l}}{p^{p^k}}  + \delta_1(g) h(k,m)
%%\end{multline*}
% $$ h(k,m)=  \frac{ (p-1) p^{m-1}(p^m-1) + (1-\delta_0(m)) (p^{m-1}-1)(p^k-1)}{p^{2k}}.$$
 \begin{multline*}
     \frac{\delta_1(g)}{p^{(1+\delta_0(m))k}}+\left(1+\delta_0(m)\frac{2}{p}\right)n^2 (p-1)^2 \frac{p^{p^{l-1}-2l}}{p^{p^k}}\leq
 P(W_k,e_n,g)\\
 \leq  4\delta_1(g)\frac{n}{p^{(1+\delta_0(m))k}} + 4 n^2 (p-1)^2 \frac{(8k+4) p^{p^{l-\delta_0(m)}-2l}}{p^{p^k}}.
 \end{multline*} 
 Note that, for the upper bound above, we have use the fact that 
 \[
    h(k,m) \le \frac{2n}{p^{(1+\delta_0(m))k}}
 \]
 for any $k$ and $m$.
		%The lower and upper bounds for $S_1(n,g)$ obtained in \cref{lemma:upperboundsS_ig} are bigger\comment{what? Cannot we put everything together?} than those for $S_0(n,g)$, and this yields the result. Now, we have 
		%\[ P(W_k,e_n)= \frac{1}{|W_k|^2}\sum_{t_1,t_2\in T} s_{t_1,t_2}(n)=\frac{S_0(n)+S_1(n)}{p^{2k+2p^k}}.\]   
		%In this case, we consider the dominant lower and upper bounds (respectively, that for $S_0(n)$ and that for $S_1(n)$) obtained in \cref{lemma:upperbuondsS_i}. This yields the result.
	\end{proof}
	
    As an application, we prove a stronger version of the Amit-Ashurst conjecture for the Engel word $e_{p^m}$ in the groups $W_k$.
    
    \begin{proof}[Proof of \cref{thm:amit_ashurts_gen}]
        We may assume $m\leq k-1$. Note that, by \cref{cor:countingkerim}, \[ p^{p^k-p^m}=|\gamma_{p^m+1}(W_k)|=|\mathrm{Im}(e_n)| \] since $\mathrm{Im}(e_n)=\gamma_{n+1}(W_k)$. Let $h$ be an element of $W_k$ of degree $p^m+1$, so that $l(h)=m+1$. By \cref{prop:biggest_fiber}, $P(W_k,e_n,g)\geq P(W_k,e_n,h)$ since $\mathrm{deg}(g)\geq\mathrm{deg}(h)$.
        
        %If $m=0$, by \cref{thm:engelbound_gen}, we have 
        %\[ P(W_k,e_1,h)\geq \frac{c(1,p)p^p}{p^4}\frac{1}{p^{p^k-1}}
        %=\biggl(1-\frac{1}{p}\biggr)^2p^{p-2} \frac{1}{|\mathrm{Im}(e_1)|}. \]
        %In particular, since $|W_k|=p^{k+1}|\mathrm{Im}(e_1)|$, it follows 
        %\[ P(W_k,e_1,h)\geq \biggl(1-\frac{1}{p}\biggr)^2p^{k+p-1}\frac{1}{|W_k|}\geq (p-1)^2p^{p-2}\frac{1}{|W_k|}\geq \frac{1}{|W_k|} \]
        %since $p\geq 2$.

        %Analogously, if $m\geq1$, we have
        By \cref{thm:engelbound_gen}, we have
        \[ P(W_k,e_n,h)\geq 
        \left(1+\delta_0(m)\frac{2}{p}\right)n^2 (p-1)^2 \frac{p^{p^{m}-2m-2}}{p^{p^k}},
        %, \left(1+\delta_0(m)\frac{2}{p}\right)\biggl(1-\frac{1}{p}\biggr)^2 \frac{1}{|\mathrm{Im}(e_n)|}. 
        \]
        and hence the first inequality follows.
        
        Now, for every prime number $p$, note that $p+2\geq 4$ and $1-\frac{1}{p}\geq \frac{1}{2}$, so that $(p+2)(1-\frac{1}{p})^2\geq 1$. Therefore, since $|B_k|=p^{n}|\mathrm{Im}(e_n)|$, we have 
        \[ P(W_k,e_n,h)\geq \left(1+\delta_0(m)\frac{2}{p}\right)\biggl(1-\frac{1}{p}\biggr)^2p^n \frac{1}{|B_k|}\geq \frac{1}{|B_k|}.
        \qedhere\]
        
    \end{proof}

	\section{Applications to pro-\texorpdfstring{$p$}{p} groups}\label{pro-psection}

    We begin with two results regarding the Haar measure of pre-images of the form $X(G,w,S):=w^{-1}(S)$, where $w$ is a word in $d$ letters and $S$ is a closed subset of a countably based profinite group $G$.
	
	\begin{lemma}\label{2-lemma:doubleseq}
		Let $G$ be a profinite group and $w\in F_d$ a non-trivial word in $d$ letters. Let $S$ be a closed subset of $G$ and $\{ N_i \}_{i\geq 1}$ a descending chain of open normal subgroups of $G$, and define $x\colon\N\times\N\to\mathbb{R}$ by
        \[
        x(i,j)=\frac{|\pi_j^d(X(G,w,SN_i))|}{|G:N_j|^d},
        \]
        where $\pi_j^d\colon G^d\to (G/N_j)^d$ is the natural projection. Then $x$ is decreasing, that is, $x(i',j')\leq x(i,j)$ for every $i'\geq i$ and $j'\geq j$.
	\end{lemma}
	\begin{proof}
		Let $i'\geq i$, $j'\geq j$, and since $N_{j'}\leq N_j$, let $\pi_{jj'}^d\colon (G/N_{j'})^d\to(G/N_j)^d$ be the natural projection. Moreover, write $X_i=X(G,w,SN_i)$ and note that $x(i,j)=\mu_j^d(\pi_j^d(X_i))$ where $\mu_j^d$ is the normalized Haar measure on $(G/N_j)^d$.
        Since $(G/N_j)^d$ is a quotient of $(G/N_{j'})^d$, we have $\mu_{j'}^d(X)\leq \mu_{j}^d(\pi_{jj'}^d(X))$ for every $X\subseteq (G/N_{j'})^d$. In particular, \[ \mu_{j'}^d(\pi_{j'}^d(X_{i'}))\leq\mu_j^d(\pi_{jj'}^d(\pi_{j'}^d(X_{i'})))=\mu_j^d(\pi_j^d(X_{i'})) \] since $\pi_{jj'}^d\circ \pi_{j'}^d=\pi_j^d$.
        Therefore, as  $X_{i'}\subseteq X_i$, it follows that \[ x(i',j')=\mu_{j'}^d(\pi_{j'}^d(X_{i'}))\leq \mu_j^d(\pi_j^d(X_{i'}))\leq\mu_j^d(\pi_j^d(X_i))=x(i,j). \qedhere\] 
	\end{proof}
	
	\begin{proposition}\label{2-prop:probformula}
		Let $G$ be a countably based profinite group an let $w\in F_d$ be a non-trivial word in $d$ letters. Let $S$ be a closed subset of $G$. Then 
		\begin{equation}\label{2-eq:2corollario}
			P(G,w_G, S)=\underset{i\to\infty}{\mathrm{lim}}P(G/N_i,w_{G/N_i},SN_i/N_i),
		\end{equation}
		for every descending chain $\{N_i\}_{i\geq 1}$ of open normal subgroups of $G$ with trivial intersection.
	\end{proposition}
	\begin{proof}
		
		Note that $P(G,w_G,S)=\mu_{G^d}(w_G^{-1}(S))$, where $\mu_{G^d}$ is the normalized Haar measure on $G^d$. Since the word map $w_G$ is continuous and $S$ is closed, we have 
		\[ X(G,w_G,S)=w_G^{-1}(S) = \bigcap_{i\geq 1} w_G^{-1}(SN_i) = \bigcap_{i\geq 1} X(G,w,SN_i). \]
		Write $x(i)=\mu_{G^d}(X(G,w,SN_i))$. By standard properties of the Haar measure, it follows that
		\[ P(G,w_G,S)=\underset{i\to\infty}{\mathrm{lim}} x(i). \]
		Moreover, by \cite[Pag. 203]{LS03}, we have $x(i)=\underset{j\to\infty}{\mathrm{lim}} x(i,j)$, where  
		\[x(i,j):=\frac{|\pi_j(X(G,w,SN_i))|}{|G:N_j|^d}. \]
		Note that $x(i,j)\in[0,1]$ for every $i,j$ and $s(i',j')\leq s(i,j)$ for every $i'\geq i$ and $j'\geq j$ by \cref{2-lemma:doubleseq}. Therefore, by \cite[Theorem 4.2]{Habil06}, we have 
		\begin{equation}\label{2-eq:doubleseqeq}
			P(G,w_G,S)=\underset{i\to\infty}{\mathrm{lim}}(\underset{j\to\infty}{\mathrm{lim}} x(i,j))=\underset{i,j}{\mathrm{inf}}\ x(i,j). 
		\end{equation} 
		
		Now, if $N\unlhd G$ and $\pi_N^d\colon G^d\to (G/N)^d$ is the natural projection, then
        $$
        \pi_N^d(X(G,w,SN))=X(G/N,w_{G/N},SN/N).
        $$
        Therefore, since the infimum in (\ref{2-eq:doubleseqeq}) is realized by the subsequence $\{s(i,i)\}_{i\geq 1}$, we obtain that 
		\[ P(G,w_G,S)=\underset{i,j\in\N}{\mathrm{inf}}x(i,j)=\underset{i\to\infty}{\mathrm{lim}} x(i,i)=\underset{i\to\infty}{\mathrm{lim}} P(G/N_i,w_{G/N_i},SN_i/N_i).\qedhere\]
	\end{proof}
	In what follows, we use the notation $W_k=C_p\wr C_{p^k}=\langle\x_k\rangle\wr\langle w_k\rangle$. For every $l\geq k$, consider the group homomorphism 
    \begin{align*}
	\vphi_{kl}\colon W_l&\to W_k\\
	 \x_l &\to \x_k\\
	 w_l &\to w_k.
    \end{align*}
    Then $(W_k,\vphi_{kl})$ is an inverse system of finite $p$-group and its inverse limit is the pro-$p$ wreath product
    \[ W:=\underset{k}{\varprojlim}\hspace{1mm} W_k= C_p\wr\Z_p=\langle\x\rangle\wr\langle w\rangle, \]  with pro-cyclic top group $\Z_p$ and elementary abelian base group $B$ given by the direct product of copies of $C_p$ indexed by $\Z_p$. Note that $W$ is a $2$-generator pro-$p$ group of infinite rank. In particular, $W$ is a finitely generated pro-$p$ group that is not $p$-adic analytic.
	
    Consider the $n$-th Engel map $e_n\colon W^2\to W$ and let $I_n=\mathrm{Im}(e_n)$.
    Recall that, if we denote by $e_{n,k}\colon W_k^2\to W_k$ the $n$-th Engel map evaluated on $W_k$, we have $I_{n,k}:=\mathrm{Im}(e_{n,k})=\gamma_{n+1}(W_k)$ (compare \cref{pr:degree}).  Since $I_n=\underset{k}{\varprojlim}\hspace{1mm} I_{n,k}$, it follows that $I_n=\gamma_{n+1}(W)$. %Moreover, given topological spaces $X,Y$ and a map $f\colon X\to Y$, recall that $f$ is a \emph{relatively open map} if whenever $U$ is an open subset of $X$, then $f(U)$ is an open subset of the image $\mathrm{Im}(f)$, endowed with the subspace topology induced by the codomain $Y$. Since $\mathrm{Im}(e_n)=\gamma_{n+1}(W)$ and $\cap_{i\geq 1}\gamma_i(W)=1$, the set \[ \mathcal{B}_n=\{ \gamma_{i+1}(W)\mid i\geq n \} \] is a base for the topology of $\mathrm{Im}(e_n)$.

	\begin{proposition}\label{prop:engelmeasure}
		Let $n=p^m$, with $m\in\N_0$, and let $S$ be a closed subset of $I_n$ such that $\mu_{I_n}(S)>0$. Then $\mu_{W^2}(e_n^{-1}(S))>0.$
	\end{proposition}
	\begin{proof}
		Let $N_k=\mathrm{Ker}(\pi_k\colon W\to W_k)$. By \cref{2-prop:probformula}, we have
		\begin{equation}\label{eq:unoprop}
			\mu_{W^2}(e_n^{-1}(S))=P(W,{e_n},S)=\underset{k\to\infty}{\mathrm{lim}}P(W_k,e_n,\pi_k(S)).
		\end{equation} 
		Moreover, by \cref{thm:amit_ashurts_gen}, there is a constant $a(n,p)>0$ such that \begin{equation}\label{eq:dueprop}
			P(W_k,e_n,\pi_k(S))\geq a(n,p)\frac{|\pi_k(S)|}{|I_{n,k}|}
		\end{equation} for every $k\geq 1$. Let $\rho_k$ be the restriction of $\pi_k$ to $I_n$ and note that $I_n/I_n\cap N_k\simeq I_{n,k}$.
		%Now, since $|B:I_n|$ is finite, we have $\mu_{I_n}(S)=|B:I_n|\mu_B(S)$. 
        In particular, 
		\begin{equation}\label{eq:treprop}
		    0<\mu_{I_n}(S)=\underset{k\to\infty}{\mathrm{lim}}\frac{|\rho_k(S)|}{|I_n:I_n\cap N_k|}=\underset{k\to\infty}{\mathrm{lim}}\frac{|\pi_k(S)|}{|I_{n,k}|}. 
		\end{equation} 
		Combining Equations (\ref{eq:unoprop}), (\ref{eq:dueprop}), and (\ref{eq:treprop}), the proposition follows.
	\end{proof}
    In the following two results, we collect some algebraic and topological properties of $e_n$.
    
	\begin{lemma}\label{lemma:engelcosetimage}
		Let $g_1=w^{t_1}x_1$, $g_2=w^{t_2}x_2$ be elements of $W_k$ and let $N\unlhd W_k$ be such that $N\subseteq B$. Then \[ e_n(g_1N,g_2N)=e_n(g_1,g_2)e_n(w^{t_1}N,w^{t_2}N). \]
	\end{lemma}
	\begin{proof}
		Let $y_1,y_2\in N$. Since $x_1y_1,x_2y_2\in B$, by \cref{lemma:decomposition} we have 
		\[ e_n(w^{t_1}x_1y_1,w^{t_2}x_2y_2)=[x_1y_1,_nw^{t_2}][x_2y_2,w^{t_1},_{n-1}w^{t_2}]^{-1}. \]
		Since the map $x\mapsto [x,w^t]$ is a group homomorphism (compare \cref{lemma:factorization}), it follows that
		\[ [x_1y_1,_nw^{t_2}][x_2y_2,w^{t_1},_{n-1}w^{t_2}]^{-1}= e_n(g_1,g_2)e_n(w^{t_1}y_1,w^{t_2}y_2), \] and hence the result.
	\end{proof}

	\begin{proposition}\label{prop:criteriumemptyinterior2}
		Let $X$ be a subset of $W$. If $e_n^{-1}(X)$ has non-empty interior, then $X$ contains a coset of $\gamma_h(W)$ for some $h\in\N$. 
	\end{proposition}
	\begin{proof}
		Let $N\unlhd_oW$ and $g_1,g_2\in W$ be such that $e_n(g_1N,g_2N)\subseteq X$. Since $W/N$ is a finite $p$-group, there is $h^*\in\N$ such that $\gamma_{h^*}(W)\subseteq N$. By \cref{lemma:engelcosetimage}, we have 
		\[ e_n(g_1\gamma_{h^*}(W),g_2\gamma_{h^*}(W))=e_n(g_1,g_2)e_n(w^{t_1}\gamma_{h^*}(W),w^{t_2}\gamma_{h^*}(W)). \]
		In particular, since $e_n(w^{t_1}\gamma_{h^*}(W),w^{t_2})=\gamma_{h^*+np^{\nu(t_2)}}(W)$ (see \cref{pr:degree}), it follows that \[ X\supseteq e_n(g_1N,g_2N)\supseteq e_n(g_1,g_2)\gamma_h(W) \] for some $h\in\N$.
	\end{proof}
	
	Note that, since $\mathrm{Im}(e_n)=\gamma_{n+1}(W)$ and $\cap_{i\geq 1}\gamma_i(W)=1$, the set \[ \mathcal{B}_n=\{ \gamma_{i+1}(W)\mid i\geq n \} \] is a base for the topology of $\mathrm{Im}(e_n)$. Combining \cref{prop:engelmeasure} and \cref{prop:criteriumemptyinterior2}, the proof of \cref{cor:construction} follows.
	
	To conclude, we show that such a closed subset $S$ of $I_n$ exists. Since $I_n$ is a closed subgroup of the $2$-generator pro-$p$ group $W$, we have that $I_n$ is a countably based pro-$p$ group. 

\begin{lemma}[Folklore]
    Let $G$ be a countably based profinite group. Then, $G$ contains a closed subset of positive Haar measure and empty-interior.
\end{lemma}
\begin{proof}
    Recall that being countably based is equivalent to the existence of a filtration $G= G(0) > G(1) > G(2) > \ldots$ of open subgroups of $G$.    
    We define inductively a subset $C \subset G$. Set $C_1=G(1)$. Passing from step $n-1$ to step $n$, the set $C_{n}$ is the union of exactly one coset of $G(\frac{n(n+1)}{2})$ in each coset of $G(\frac{n(n-1)}{2})$ that has not been included in $C_n$. Set $C=\bigcup_{n=1}^\infty C_n$.

    Note that $C$ is open, being a union of open sets, and it is easy to see that it is dense in $G$, because $\bigcap
    G(n)$ is trivial.
    %\commatteo{Take an open set $U\subseteq G$. Then there is a coset of some $G_{n(n-1)/2}$ contained in $U$ (for some $n$). But then an least one coset of $G_{n(n+1)/2}$ is in $U$.}
    Moreover, we can estimate the measure of $C$ in the following way:
    \begin{align*}
      \mu(C)& = \frac{1}{n_1} + \frac{n_1-1}{n_1 n_2 n_3} + \frac{(n_1-1)(n_2 n_3-1)}{n_1 \ldots n_6} + \cdots 
      \\
      &+  \frac{(n_1-1)(n_2 n_3-1)(n_4 n_5 n_6 -1)\ldots(n_{\frac{(k-2)(k-1)}{2}+1} \ldots n_{\frac{(k-1)k}{2}} -1)}{n_1 \ldots n_{\frac{k(k+1)}{2}}} + \cdots\\
      & <
      \sum_{k=1}^\infty \frac{1}{n_{\frac{(k-1)k}{2}+1} \ldots n_{\frac{k(k+1)}{2}}}
      %\frac{1}{n_1}+ \frac{1}{n_2^2}+ \frac{1}{n_3^3}+ \ldots 
      \le \sum_{k=1}^\infty \frac{1}{2^k} = \frac{1}{1-\frac{1}{2}}-1 =1.
    \end{align*}
    Therefore, the complement of $C$ is the closed set of positive measure and empty interior we were searching for.\qedhere

    % \textcolor{blue}{$$ \displaystyle\sum_{k=1}^\infty \frac{\displaystyle \prod_{i=1}^k \left( \left(\prod_{j=\frac{(i-1)i}{2}+1}^{\frac{i(i+1)}{2}} n_j\right) -1\right) }{\displaystyle \prod_{i=1}^{\frac{k(k+1)}{2}} n_i } $$ this is very ugly, so I would rather write as in the magenta colored part...}
\end{proof}

\end{document}